\documentclass[a4paper,twoside,english,10pt]{amsart}

\usepackage{fullpage}
\usepackage{amsmath,amsfonts,amssymb}
\usepackage[all]{xy}
\usepackage[english]{babel}
\usepackage[utf8x]{inputenc}
\usepackage{url}
\usepackage{ifthen}
\usepackage{booktabs}
\usepackage{multirow}
\usepackage{array}
\usepackage{paralist}
\usepackage{enumerate}
\usepackage{graphicx}
\usepackage{rotating}

\usepackage[hypertexnames=false,
backref=page,
    pdftex,
    pdfpagemode=UseNone,
    breaklinks=true,
    extension=pdf,
    colorlinks=true,
    linkcolor=blue,
    citecolor=blue,
    urlcolor=blue,
]{hyperref}

\newcommand{\referenza}{}

 \newtheorem{thm}{Theorem}[section]
 \newtheorem*{thm*}{Theorem {\referenza}}
 \newtheorem{prop}[thm]{Proposition}
 \newtheorem{cor}[thm]{Corollary}
 \newtheorem*{cor*}{Corollary {\referenza}}

\theoremstyle{definition}
 \newtheorem{defi}[thm]{Definition}
 \newtheorem{rem}[thm]{Remark}
 \newtheorem{ex}[thm]{Example}
 \newtheorem{lem}[thm]{Lemma}

\newcommand{\K}{\mathbb{K}}
\newcommand{\N}{\mathbb{N}}
\newcommand{\Z}{\mathbb{Z}}
\newcommand{\R}{\mathbb{R}}
\newcommand{\C}{\mathbb{C}}

\newcommand{\st}{\;:\;}
\newcommand{\sspace}{{\cdot}}
\newcommand{\ssspace}{{\cdot\cdot}}
\newcommand{\del}{\partial}
\newcommand{\delbar}{\overline{\del}}

\newcommand{\g}{\mathfrak{g}}
\newcommand{\solvmfd}{\left. \Gamma \middle\backslash G \right.}

\renewcommand{\Im}{\mathsf{Im}}
\renewcommand{\Re}{\mathsf{Re}}

\DeclareMathOperator{\id}{id}
\DeclareMathOperator{\imm}{im}
\DeclareMathOperator{\de}{d}
\DeclareMathOperator{\esp}{e}
\DeclareMathOperator{\GL}{GL}

\allowdisplaybreaks[1]

\title{Hodge theory for twisted differentials}

\author{Daniele Angella}
\address[Daniele Angella]{Istituto Nazionale di Alta Matematica}
\curraddr{Dipartimento di Matematica e Informatica\\
Universit\`{a} di Parma \\
Parco Area delle Scienze 53/A, 43124 \\
Parma, Italy}
\email{daniele.angella@math.unipr.it}

\author{Hisashi Kasuya}
\address[Hisashi Kasuya]{Department of Mathematics\\
Tokyo Institute of Technology\\
1-12- 1-H-7, O-okayama, Meguro\\
Tokyo 152-8551, Japan}
\email{khsc@ms.u-tokyo.ac.jp}
\email{kasuya@math.titech.ac.jp}

\thanks{The first author is granted with a research fellowship by Istituto Nazionale di Alta Matematica INdAM, and is supported by the Project PRIN ``Varietà reali e complesse: geometria, topologia e analisi armonica'', by the Project FIRB ``Geometria Differenziale e Teoria Geometrica delle Funzioni'', and by GNSAGA of INdAM. The second author is supported by the JSPS Grant-in-Aid for Research Activity start-up}

\begin{document}

\begin{abstract}
 We study cohomologies and Hodge theory for complex manifolds with twisted differentials. In particular, we get another cohomological obstruction for manifolds in class $\mathcal{C}$ of Fujiki. We give a Hodge-theoretical proof of the characterization of solvmanifolds in class $\mathcal{C}$ of Fujiki, first proven by D. Arapura.
\end{abstract}

\maketitle

\section*{Introduction}

One of the possible ways of generalizing K\"ahler structures is to consider Hermitian metrics being locally conformal to a K\"ahler metric; see, e.g., \cite{dragomir-ornea}.

More precisely, consider a complex manifold $(M,J)$. Suppose that it admits a $J$-Hermitian metric $g$ being locally conformal to a K\"ahler metric. That is, for every point $p\in M$, there exist an open neighbourhood $U\ni p$ in $M$ and a smooth function $f\in\mathcal{C}^\infty(U;\R)$ such that $g=\exp(-f)\,\tilde g$, where $\tilde g$ is K\"ahler. Consider the associated $(1,1)$-forms $\omega:=g(J\sspace, \ssspace)$ and $\tilde\omega:=\tilde g(J\sspace, \ssspace)=\exp(f)\,\omega$. Since $\tilde g$ is K\"ahler, it follows that $\omega$ satisfies
$$ \de \omega + \de f \wedge \omega \;=\; 0 \;. $$
In fact, by the Poincaré Lemma, the property of $g$ being locally conformal K\"ahler is characterized by the existence of a $\de$-closed $1$-form $\phi \in A^1(M)_\R$ such that
$$ \de_\phi \omega \;=\;0 \qquad \text{ where } \qquad \de_\phi \;:=\; \de + L_\phi \;, $$
where $L_\phi:=\phi\wedge\sspace$.

The cochain complex $(A^{\bullet}(M)_\R,\de_{\phi})$ can be considered as the de Rham complex with values in the topologically trivial flat bundle $M\times \R$ with the  connection form $\phi$. Hence it is determined by the character $\rho_{\phi} \colon \pi_{1}(M)\to \mathrm{GL}_{1}(\K)$ given by $\rho_\phi(\gamma)=\exp\left({\int_{\gamma}\phi}\right)$. In particular, it is determined by the cohomology class $[\phi]\in H^{1}(M;\R)$.

\medskip

On compact K\"ahler manifolds, the Hodge theory for local systems was developed by the theory of Higgs bundles, see \cite{Sim}.
In this note, we study cohomologies and Hodge theory for general complex manifolds $(M, J)$ with twisted differentials. More precisely, denote by $H_{BC}^{\bullet,\bullet}(M):=\frac{\ker\del\cap\ker\delbar}{\imm\del\delbar}$ the Bott-Chern cohomology of $(M, J)$; in particular, $H_{BC}^{1,0}(M)=\{\theta\in A^{1,0}(M)\vert \del\theta=\delbar\theta=0\}$. For $\theta_1,\theta_2\in H^{1,0}_{BC}(M)$, consider the bi-differential $\Z$-graded complex
$$ \left( A^{\bullet}(M)_{\C},\, \del_{(\theta_{1},\theta_{2})},\, \delbar_{(\theta_{1},\theta_{2})} \right) \;, $$
where
$$ \del_{(\theta_{1},\theta_{2})} \;:=\; \partial+L_{\theta_{2}}+L_{\overline{\theta_{1}}} \qquad \text{ and } \qquad \delbar_{(\theta_{1},\theta_{2})} \;:=\; \delbar -L_{\overline{\theta_{2}}}+L_{\theta_{1}} \;. $$

Several cohomologies can be defined, and the identity induces natural maps between them:
$$ \xymatrix{
 & H^{\bullet}\left(A^{\bullet}(M)_{\C};\, \del_{(\theta_{1},\theta_{2})},\, \delbar_{(\theta_{1},\theta_{2})};\, \del_{(\theta_{1},\theta_{2})}\delbar_{(\theta_{1},\theta_{2})}\right) \ar[ld]_{\iota_{BC,\del}} \ar[d]_{\iota_{BC,dR}} \ar[rd]^{\iota_{BC,\delbar}} \ar@/_5pc/[dd]_{\iota_{BC,A}} & \\
 H^{\bullet}\left(A^{\bullet}(M)_{\C};\, \del_{(\theta_{1},\theta_{2})}\right) \ar[rd]_{\iota_{\del,A}} & H^{\bullet}\left(A^{\bullet}(M)_{\C};\, \de_{\phi}\right) \ar[d]^{\iota_{dR,A}} & H^{\bullet}\left(A^{\bullet}(M)_{\C};\, \delbar_{(\theta_{1},\theta_{2})}\right) \ar[ld]^{\iota_{\delbar,A}} \\
 & H^{\bullet}\left(A^{\bullet}(M)_{\C};\, \del_{(\theta_{1},\theta_{2})}\delbar_{(\theta_{1},\theta_{2})};\, \del_{(\theta_{1},\theta_{2})},\, \delbar_{(\theta_{1},\theta_{2})}\right) &
 } $$

 Here, $H^{\bullet}\left(A^{\bullet}(M)_{\C};\, \de_{\phi}\right)$, $H^{\bullet}\left(A^{\bullet}(M)_{\C};\, \del_{(\theta_{1},\theta_{2})}\right)$, and $H^{\bullet}\left(A^{\bullet}(M)_{\C};\, \delbar_{(\theta_{1},\theta_{2})}\right)$ denote the cohomology of the corresponding complex, and, in the notation of \cite{deligne-griffiths-morgan-sullivan},
 \begin{eqnarray*}
 H^{\bullet}\left(A^{\bullet}(M)_{\C};\, \del_{(\theta_{1},\theta_{2})},\, \delbar_{(\theta_{1},\theta_{2})};\, \del_{(\theta_{1},\theta_{2})}\delbar_{(\theta_{1},\theta_{2})}\right) &:=& \frac{\ker \del_{(\theta_{1},\theta_{2})}\cap \ker \delbar_{(\theta_{1},\theta_{2})}}{\imm \del_{(\theta_{1},\theta_{2})}\delbar_{(\theta_{1},\theta_{2})}} \;, \\[5pt]
 H^{\bullet}\left(A^{\bullet}(M)_{\C};\, \del_{(\theta_{1},\theta_{2})}\delbar_{(\theta_{1},\theta_{2})};\, \del_{(\theta_{1},\theta_{2})},\, \delbar_{(\theta_{1},\theta_{2})}\right) &:=& \frac{\ker \del_{(\theta_{1},\theta_{2})}\delbar_{(\theta_{1},\theta_{2})}}{\imm \del_{(\theta_{1},\theta_{2})} + \imm \delbar_{(\theta_{1},\theta_{2})}} \;,
\end{eqnarray*}
 are the counterpart of Bott-Chern \cite{bott-chern} and Aeppli \cite{aeppli} cohomologies.

 \medskip
 
 The above maps are in general neither injective nor surjective: so they do not allow a direct comparison of the cohomologies.
 Hence we are especially interested in studying the following properties:
 \begin{itemize}
  \item $(M, J)$ is said to satisfy the {\em $\del_{(\theta_1,\theta_2)}\delbar_{(\theta_1,\theta_2)}$-Lemma} if the natural map $\iota_{BC,A}$ is injective;
  \item $(M, J)$ admits the {\em $\left(\theta_1, \theta_2\right)$-Hodge decomposition} if the natural maps $\iota_{BC,\del}$ and $\iota_{BC,dR}$ and $\iota_{BC,\delbar}$ are isomorphisms.
 \end{itemize}
 Admitting the $\left(\theta_1, \theta_2\right)$-Hodge decomposition is a stronger property than satisfying the $\del_{(\theta_1,\theta_2)}\delbar_{(\theta_1,\theta_2)}$-Lemma. For $(\theta_1,\theta_2)=(0,0)$, the above properties are in fact equivalent. See \cite{deligne-griffiths-morgan-sullivan} for a proof. If $\left( M ,\, J \right)$ admits K\"ahler metrics, then it satisfies the two properties for any $\theta_{1}, \theta_{2} \in H^{1,0}_{BC}(M)$.

\renewcommand{\referenza}{\ref{cor:kahler-deldelbar-hodge-tw}}
\begin{cor*}
 Let $\left( M ,\, J \right)$ be a compact complex manifold endowed with a K\"ahler metric.
 Take $\theta_{1}, \theta_{2} \in H^{1,0}_{BC}(M)$.
 Then $\left( M ,\, J \right)$ satisfies the $\del_{(\theta_1,\theta_2)}\delbar_{(\theta_1,\theta_2)}$-Lemma and admits the $(\theta_1,\theta_2)$-Hodge decomposition.
\end{cor*}

 \medskip

 In Example \ref{ex:nakamura}, we study an explicit example on the completely-solvable Nakamura manifold $X=\solvmfd$.
 It is known that $\solvmfd$ satisfies the $\del\delbar$-Lemma, see \cite[Example 2.17]{angella-kasuya-1}.
 For $\theta_1:=\frac{\de z_1}{2} \in H^{1,0}_{BC}(X)$ and $\theta_2:=0$, (see page \pageref{ex:nakamura} for notation,) it follows that $X$ does not admit the $(\theta_1,\theta_2)$-Hodge decomposition, see \cite[\S8]{kasuya-hyper-hodge}. Furthermore, we show that it does not satisfy the $\del_{(\theta_1,\theta_2)}\delbar_{(\theta_1,\theta_2)}$-Lemma.

\medskip

 In particular, we are interested in studying the behaviour of $\del_{(\theta_1,\theta_2)}\delbar_{(\theta_1,\theta_2)}$-Lemma and $\left(\theta_1, \theta_2\right)$-Hodge decomposition under modifications. We recall that a {\em modification} $\mu \colon \left( \tilde M ,\, \tilde J \right) \to \left( M ,\, J \right)$ is a holomorphic map between compact complex manifolds of the same dimension that yields a biholomorphism $\mu\lfloor_{\tilde M\backslash \mu^{-1}(S)} \colon \tilde M\backslash \mu^{-1}(S) \to  M\backslash S$ outside the preimage of an analytic subset $S \subset M$ of codimension greater than or equal to $1$. We prove the following results.

\renewcommand{\referenza}{\ref{thm:deldelbar-modification} and Theorem \ref{thm:hodgedec-modification}}
\begin{thm*}
 Let $\mu \colon \left( \tilde M ,\, \tilde J \right) \to \left( M ,\, J \right)$ be a proper modification of a compact complex manifold $\left( M,\, J \right)$.
 Take $\theta_{1}, \theta_{2} \in H^{1,0}_{BC}(M)$.
 \begin{itemize}
  \item If $\left( \tilde M ,\, \tilde J \right)$ satisfies the $\del_{(\mu^{\ast}\theta_{1},\mu^{\ast}\theta_{2})}\delbar_{(\mu^{\ast}\theta_{1},\mu^{\ast}\theta_{2})}$-Lemma, then $\left( M ,\, J \right)$ satisfies the $\del_{(\theta_{1},\theta_{2})}\delbar_{(\theta_{1},\theta_{2})}$-Lemma.
  \item If $\left( \tilde M ,\, \tilde J \right)$ satisfies the $\left(\mu^{\ast}\theta_{1},\mu^{\ast}\theta_{2}\right)$-Hodge decomposition, then $\left( M ,\, J \right)$ satisfies the $\left(\theta_{1},\theta_{2}\right)$-Hodge decomposition.
 \end{itemize}
\end{thm*}

 From this, we get another cohomological obstruction for complex manifolds belonging to class $\mathcal{C}$ of Fujiki. We recall that a compact complex manifold $\left( M,\, J \right)$ is said to be in {\em class $\mathcal{C}$ of Fujiki} \cite{fujiki} if it admits a proper modification $\mu \colon \left( \tilde M ,\, \tilde J \right) \to \left( M ,\, J \right)$ with $\left( \tilde M ,\, \tilde J \right)$ admitting K\"ahler metrics.

\renewcommand{\referenza}{\ref{cor:classC-hodge}}
\begin{cor*}
 Let $\left( M ,\, J \right)$ be a compact complex manifold in class $\mathcal{C}$ of Fujiki.
 Take $\theta_{1}, \theta_{2} \in H^{1,0}_{BC}(M)$.
 Then $\left( M ,\, J \right)$ satisfies the $\del_{(\theta_{1},\theta_{2})}\delbar_{(\theta_{1},\theta_{2})}$-Lemma and admits the $\left(\theta_1,\theta_2\right)$-Hodge decomposition.
\end{cor*}

The previous results can be adapted to compact complex orbifolds of global-quotient type, namely, quotients of compact complex manifolds by finite groups of biholomorphisms, see Theorem \ref{thm:mod-orbfd-glob} and Corollary \ref{cor:orbld-glob-class-c}.

\medskip

The second author studied in \cite{kasuya-hyper-hodge} the property of satisfying the $(\theta_1,\theta_2)$-Hodge decomposition for any $\theta_1,\theta_2\in H^{1,0}_{BC}(M)$. In \cite[Theorem 1.7]{kasuya-hyper-hodge}, he proved that a solvmanifold admitting hyper-strong-Hodge-decomposition admits a K\"ahler metric. Therefore, we get a more direct proof of the characterization of solvmanifolds in class $\mathcal{C}$ of Fujiki.

\renewcommand{\referenza}{\ref{thm:classC-solvmanifolds}}
\begin{thm*}[{see also \cite[Theorem 9]{arapura}, \cite[Theorem 1.1]{bharali-biswas-mj}}]
 Let $\left( M ,\, J \right)$ be a solvmanifold endowed with a complex structure. If $\left( M ,\, J \right)$ belongs to class $\mathcal{C}$ of Fujiki, then it admits a K\"ahler metric.
\end{thm*}

This result was firstly proven by D. Arapura \cite{arapura},
by using the fact that ``the classes of fundamental groups of compact [manifolds in class $\mathcal{C}$ of Fujiki] and compact K\"ahler manifolds coincide'',
which is proven by Hironaka elimination of indeterminacies (see \cite[Lemma 2.1]{bharali-biswas-mj}).
But our proof does not rely on the Hironaka elimination of indeterminacies.

\bigskip

\noindent{\sl Acknowledgments.}
The first author is greatly indebted to Adriano Tomassini for his constant support and for many useful discussions.

\section{Twisted differentials and cohomologies on complex manifolds}

\subsection{Twisted differentials}

Let $(M,\, J)$ be a complex manifold of complex dimension $n$. For $\K \in \{ \R, \C \}$, denote by $A^{\bullet}(M)_{\K}$ the space of $\K$-valued differential forms on $M$. Consider the cochain complex $\left( A^\bullet(M)_\K,\, \de \right)$.

For $\phi\in A^{r}(M)_{\K}$, we define the operator
$$ L_{\phi}\colon A^{\bullet}(M)_{\K}\to A^{\bullet+r}(M)_{\K} \;, \qquad L_{\phi}(x) \;:=\; \phi\wedge x \;. $$
If $\phi$ is a $\de$-closed $1$-form, then the operator
$$ \de_{\phi} \;:=\; \de + L_{\phi} $$
satisfies $\de_{\phi}\circ \de_{\phi}=0$. Hence we have the cochain complex
$$ \left( A^{\bullet}(M)_{\K},\, \de_{\phi} \right) \;. $$
Note that $\de_\phi$ satisfies the following Leibniz rule:
$$ \text{ for any } \alpha\in A^{\bullet}(M)_\K \;, \quad \left[ \de_\phi, \, L_\alpha \right] \;=\; L_{\de\alpha} \;. $$

The cochain complex $(A^{\bullet}(M)_{\K},\de_{\phi})$ is considered as the de Rham complex with values in the topologically trivial flat bundle $M\times \K$ with the  connection form $\phi$.
Hence the structure of the cochain complex $(A^{\bullet}(M)_{\K},\de_{\phi})$ is determined by the character $\rho_{\phi} \colon \pi_{1}(M)\to \mathrm{GL}_{1}(\K)$ given by $\rho_\phi(\gamma)=\exp\left({\int_{\gamma}\phi}\right)$.
In particular, the cochain complex $(A^{\bullet}(M)_{\K},\de_{\phi})$ is determined by the cohomology class $[\phi]\in H^{1}(M;\K)$.

We consider the bi-grading $A^{\bullet}(M)_{\C}=A^{\bullet,\bullet}(M)$ and the decomposition $\de=\del+\delbar$.
Take $\theta_{1}, \theta_{2} \in H^{1,0}_{BC}(M)$.
Consider
$$ \del_{(\theta_{1},\theta_{2})} \;:=\; \partial+L_{\theta_{2}}+L_{\overline{\theta_{1}}} \qquad \text{ and } \qquad \delbar_{(\theta_{1},\theta_{2})} \;:=\; \delbar -L_{\overline{\theta_{2}}}+L_{\theta_{1}} \;. $$
We have
$$ \del_{(\theta_{1},\theta_{2})}\del_{(\theta_{1},\theta_{2})} \;=\; \delbar_{(\theta_{1},\theta_{2})}\delbar_{(\theta_{1},\theta_{2})} \;=\; \del_{(\theta_{1},\theta_{2})}\delbar_{(\theta_{1},\theta_{2})}+\delbar_{(\theta_{1},\theta_{2})}\del_{(\theta_{1},\theta_{2})} \;=\; 0 \;. $$
Therefore we have the bi-differential $\Z$-graded complex
$$ \left( A^{\bullet}(M)_{\C},\, \del_{(\theta_{1},\theta_{2})},\, \delbar_{(\theta_{1},\theta_{2})} \right) \;. $$
Note that $\del_{\left(\theta_{1},\theta_{2}\right)}$ and $\delbar_{\left(\theta_{1},\theta_{2}\right)}$ satisfy the following Leibniz rule:
$$ \text{ for any } \alpha\in A^{\bullet}(M)_\C \;, \quad \left[ \del_{\left( \theta_1, \theta_2 \right)}, \, L_\alpha \right] \;=\; L_{\del\alpha} \quad \text{ and } \quad \left[ \delbar_{\left( \theta_1, \theta_2 \right)}, \, L_\alpha \right] \;=\; L_{\delbar\alpha} \;. $$
Note also that the associated total cochain complex is
$$ \left( A^{\bullet}(M)_{\C},\, \de_{\left( \theta_1+\bar\theta_1 \right) + \left( \theta_2-\bar\theta_2 \right)} \right) \;. $$

\subsection{Hodge theory with twisted differentials}

Let $\left( M,\, J \right)$ be a compact complex manifold of complex dimension $n$. Take a $J$-Hermitian metric $g$ on $M$.
We consider the ($\R$-linear, possibly $\C$-anti-linear) Hodge-$*$-operator $\overline{\ast}\colon A^{\bullet}(M)_{\K}\to A^{2n-\bullet}(M)_{\K}$ associated to $g$.
Consider the inner product on $A^{\bullet}(M)_{\K}$ given by
$$ (x,y) \;:=\; \int_M x\wedge \overline{\ast}y \;. $$
Consider the adjoint operators $\de^{\ast}$, $\del^{\ast}$, and $\delbar^{\ast}$ of the operators $\de$, $\del$, and $\delbar$, respectively, with respect to $\left(\sspace,\ssspace\right)$.
Then one has
\[ \de^{\ast} \;=\; -\overline{\ast}\, \de\, \overline{\ast} \;,
\qquad \del^{\ast} \;=\; -\overline{\ast}\, \del\, \overline{\ast} \;,
\qquad \delbar^{\ast} \;=\; -\overline{\ast}\, \delbar\, \overline{\ast} \;. \]

For $\phi\in A^{r}(M)_{\K}$, consider the operator $L_{\phi}$ and define its adjoint operator with respect to $\left(\sspace,\ssspace\right)$:
$$ \Lambda_{\phi} \colon A^{\bullet}(M)_{\K}\to A^{\bullet-r}(M)_{\K}
 \quad \text{ given by } \quad
 \left( L_{\phi} \sspace, \, \ssspace \right) \;=\; \left( \sspace ,\, \Lambda_{\phi} \ssspace \right) \;. $$
For $x\in A^{|y|-r}(M)_{\K}$ and $y\in A^{|y|}(M)_{\K}$, we compute
\begin{eqnarray*}
\left( L_{\phi}x,\, y\right) &=& \int_M \phi\wedge x\wedge \overline{\ast}y \;=\; (-1)^{r\left(|y|-r\right)} \int_M x\wedge \overline{\ast}\, \overline{\ast}^{-1}\,\phi\wedge\overline{\ast}\,y \\[5pt]
&=& (-1)^{r\left(|y|-r\right)}\left(x, \, \overline{\ast}^{-1}\,L_{\phi}\,\overline{\ast}\,y \right) \;.
\end{eqnarray*}
Hence we get
$$ \left. \Lambda_{\phi} \right\lfloor_{A^{|y|}(M)_\K} \;=\; (-1)^{r(\vert y\vert-r)}\, \overline{\ast}^{-1}\, L_{\phi}\, \overline{\ast} \;. $$
In particular, since the real dimension of $M$ is even, when $\phi$ is a $1$-form, we have
$$ \Lambda_{\phi} \;=\; \overline{\ast}\, L_{\phi}\, \overline{\ast} \;. $$

For a $\de$-closed  $1$-form $\phi$, by considering the differential $\de_{\phi}=\de+\phi$, the adjoint operator $\de_{\phi}^{\ast}$ with respect to $\left(\sspace,\ssspace\right)$ is given by
\[ \de_{\phi}^{\ast} \;=\; \de^{\ast}+\Lambda_{\phi} \;=\; -\overline{\ast}\, \de\, \overline{\ast} + \overline{\ast} \, L_{\phi} \, \overline{\ast} \;=\; -\overline{\ast}\, \de_{-\phi} \, \overline{\ast} \;. \]
Analogously, for $\theta_1,\theta_2\in H^{1,0}_{BC}(M)$, the adjoint operators $\del_{\left(\theta_1,\theta_2\right)}^{\ast}$ and $\delbar_{\left(\theta_1,\theta_2\right)}^{\ast}$ of the operators $\del_{\left(\theta_1,\theta_2\right)}$ and $\delbar_{\left(\theta_1,\theta_2\right)}$, respectively, with respect to $\left(\sspace,\ssspace\right)$ are
$$ \del_{\left(\theta_1,\theta_2\right)}^{\ast} \;=\; \del^\ast + \Lambda_{\theta_2} + \Lambda_{\bar\theta_1} \;=\; - \bar\ast\,\del_{\left(-\theta_2,-\theta_1\right)}\,\bar\ast $$
and
$$ \delbar_{\left(\theta_1,\theta_2\right)}^{\ast} \;=\; \delbar^\ast - \Lambda_{\bar\theta_2} + \Lambda_{\theta_1} \;=\; - \bar\ast\,\delbar_{\left(-\theta_2,-\theta_1\right)}\,\bar\ast \;. $$

\medskip

Suppose $g$ is a K\"ahler metric, with associated K\"ahler form $\omega$.
Then we have the K\"ahler identities
\[ \Lambda_{\omega} \, \del - \del \, \Lambda_{\omega} \;=\; \sqrt{-1}\, \delbar^{\ast} \quad \text{ and } \quad \Lambda_{\omega} \, \delbar - \delbar \, \Lambda_{\omega} \;=\; -\sqrt{-1}\, \del^{\ast} \;. \]

For a  $(1,0)$-form $\theta$, as the local argument for the K\"ahler identities, see, e.g., \cite[Lemma 6.6]{Voi}, we have
\begin{eqnarray*}
\Lambda_{\omega} \, L_{\theta} - L_{\theta} \, \Lambda_{\omega} &=& \sqrt{-1}\, \Lambda_{-\bar \theta} \;=\; -\sqrt{-1}\,\Lambda_{\bar \theta} \;, \\[5pt]
\Lambda_{\omega}L_{\bar\theta}-L_{\bar\theta}\Lambda_{\omega} &=& -\sqrt{-1}\,\Lambda_{- \theta} \;=\; \sqrt{-1}\Lambda_{ \theta} \;.
\end{eqnarray*}
Hence, for $\theta_{1}, \theta_{2} \in H^{1,0}_{BC}(M)$, we have
\begin{eqnarray*} 
 \Lambda_{\omega}\del_{(\theta_{1},\theta_{2})}-\del_{(\theta_{1},\theta_{2})}\Lambda_{\omega} &=& \sqrt{-1}\, \delbar^{\ast}_{(\theta_{1},\theta_{2})} \;, \\[5pt]
 \Lambda_{\omega}\delbar_{(\theta_{1},\theta_{2})}-\delbar_{(\theta_{1},\theta_{2})}\Lambda_{\omega} &=& -\sqrt{-1}\, \del^{\ast}_{(\theta_{1},\theta_{2})} \;. 
\end{eqnarray*}

\subsection{Cohomologies with twisted differentials}

Let $(M,\, J)$ be a complex manifold; suppose also that $M$ is compact.
For $\theta_{1}, \theta_{2} \in H^{1,0}_{BC}(M)$ and $\phi:=\theta_{1}+\overline{\theta_{1}}+\theta_{2}-\overline{\theta_{2}}$,  we consider the (bi-)differential $\Z$-graded algebras
\[ \left( A^{\bullet}(M)_{\C},\, \de_\phi \right) \qquad \text{ and } \qquad \left( A^{\bullet}(M)_{\C},\, \del_{(\theta_{1},\theta_{2})},\, \delbar_{(\theta_{1},\theta_{2})} \right) \]
as above.
We consider the following cohomologies:
\begin{eqnarray*}
 H^{\bullet}\left(A^{\bullet}(M)_{\C};\, \de_{\phi}\right) &:=& \frac{\ker \de_{\phi}}{\imm \de_{\phi}} \;, \\[5pt]
 H^{\bullet}\left(A^{\bullet}(M)_{\C};\, \del_{(\theta_{1},\theta_{2})}\right) &:=& \frac{\ker \del_{(\theta_{1},\theta_{2})}}{\imm \del_{(\theta_{1},\theta_{2})}} \;, \\[5pt]
 H^{\bullet}\left(A^{\bullet}(M)_{\C};\, \delbar_{(\theta_{1},\theta_{2})}\right) &:=& \frac{\ker \delbar_{(\theta_{1},\theta_{2})}}{\imm \delbar_{(\theta_{1},\theta_{2})}} \;, \\[5pt]
 H^{\bullet}\left(A^{\bullet}(M)_{\C};\, \del_{(\theta_{1},\theta_{2})},\, \delbar_{(\theta_{1},\theta_{2})};\, \del_{(\theta_{1},\theta_{2})}\delbar_{(\theta_{1},\theta_{2})}\right) &:=& \frac{\ker \del_{(\theta_{1},\theta_{2})}\cap \ker \delbar_{(\theta_{1},\theta_{2})}}{\imm \del_{(\theta_{1},\theta_{2})}\delbar_{(\theta_{1},\theta_{2})}} \;, \\[5pt]
 H^{\bullet}\left(A^{\bullet}(M)_{\C};\, \del_{(\theta_{1},\theta_{2})}\delbar_{(\theta_{1},\theta_{2})};\, \del_{(\theta_{1},\theta_{2})},\, \delbar_{(\theta_{1},\theta_{2})}\right) &:=& \frac{\ker \del_{(\theta_{1},\theta_{2})}\delbar_{(\theta_{1},\theta_{2})}}{\imm \del_{(\theta_{1},\theta_{2})} + \imm \delbar_{(\theta_{1},\theta_{2})}} \;. \\[5pt]
\end{eqnarray*}

The identity induces natural maps
$$ \xymatrix{
 & H^{\bullet}\left(A^{\bullet}(M)_{\C};\, \del_{(\theta_{1},\theta_{2})},\, \delbar_{(\theta_{1},\theta_{2})};\, \del_{(\theta_{1},\theta_{2})}\delbar_{(\theta_{1},\theta_{2})}\right) \ar[ld] \ar[d] \ar[rd] & \\
 H^{\bullet}\left(A^{\bullet}(M)_{\C};\, \del_{(\theta_{1},\theta_{2})}\right) \ar[rd] & H^{\bullet}\left(A^{\bullet}(M)_{\C};\, \de_{\phi}\right) \ar[d] & H^{\bullet}\left(A^{\bullet}(M)_{\C};\, \delbar_{(\theta_{1},\theta_{2})}\right) \ar[ld] \\
 & H^{\bullet}\left(A^{\bullet}(M)_{\C};\, \del_{(\theta_{1},\theta_{2})}\delbar_{(\theta_{1},\theta_{2})};\, \del_{(\theta_{1},\theta_{2})},\, \delbar_{(\theta_{1},\theta_{2})}\right) &
 } $$
of $\Z$-graded vector spaces.

\medskip

By \cite[Theorem 2.4]{angella-tomassini-5}, and using the finite-dimensionality of $H^{\bullet}(A^{\bullet}(M)_{\C};\, \del_{(\theta_{1},\theta_{2})})$ and $H^{\bullet}(A^{\bullet}(M)_{\C};\, \delbar_{(\theta_{1},\theta_{2})})$, see the next subsection or \cite[page 22]{Sim}, we have the following inequality {\itshape à la} Fr\"olicher.

\begin{thm}
 Let $\left( M,\, J \right)$ be a compact complex manifold of complex dimension $n$.
 Take $\theta_{1}, \theta_{2} \in H^{1,0}_{BC}(M)$, and $\phi:=\theta_{1}+\overline{\theta_{1}}+\theta_{2}-\overline{\theta_{2}}$.
 Then the inequality
 \begin{eqnarray*}
 \lefteqn{\dim_\C H^{\bullet}(A^{\bullet}(M)_{\C};\, \del_{(\theta_{1},\theta_{2})},\, \delbar_{(\theta_{1},\theta_{2})};\, \del_{(\theta_{1},\theta_{2})}\delbar_{(\theta_{1},\theta_{2})})} \\[5pt]
 && + \dim_\C H^{\bullet}(A^{\bullet}(M)_{\C};\, \del_{(\theta_{1},\theta_{2})}\delbar_{(\theta_{1},\theta_{2})};\, \del_{(\theta_{1},\theta_{2})}, \delbar_{(\theta_{1},\theta_{2})}) \\[5pt]
 &\geq& \dim_\C H^{\bullet}(A^{\bullet}(M)_{\C};\, \del_{(\theta_{1},\theta_{2})}) + \dim_\C H^{\bullet}(A^{\bullet}(M)_{\C};\, \delbar_{(\theta_{1},\theta_{2})})
 \end{eqnarray*}
 holds.
 \end{thm}

\begin{rem}
Note that, if $\theta_{1}=0$, then $\left( A^{\bullet}(M)_{\C},\, \del_{(0,\theta_{2})},\, \delbar_{(0,\theta_{2})} \right)$ has in fact a structure of double complex. Hence we have the Fr\"olicher inequalities, \cite[Theorem 2]{frolicher},
$$ \dim_\C H^{\bullet}\left(A^{\bullet}(M)_{\C};\, \del_{(0,\theta_{2})}\right) \geq\dim_\C H^{\bullet}(A^{\bullet}(M)_{\C};\,\de_{\phi}) $$
and
$$ \dim_\C H^{\bullet}\left(A^{\bullet}(M)_{\C};\, \delbar_{(0,\theta_{2})}\right) \geq\dim_\C H^{\bullet}(A^{\bullet}(M)_{\C};\,\de_{\phi}) \;.
$$
Hence we get the following inequality {\itshape à la} Fr\"olicher, \cite[Corollary 2.6]{angella-tomassini-5},
\begin{eqnarray*}
 \lefteqn{\dim_\C H^{\bullet}(A^{\bullet}(M)_{\C};\, \del_{(\theta_{1},\theta_{2})},\, \delbar_{(\theta_{1},\theta_{2})};\, \del_{(\theta_{1},\theta_{2})}\delbar_{(\theta_{1},\theta_{2})})} \\[5pt]
 && + \dim_\C H^{\bullet}(A^{\bullet}(M)_{\C};\, \del_{(\theta_{1},\theta_{2})}\delbar_{(\theta_{1},\theta_{2})};\, \del_{(\theta_{1},\theta_{2})}, \delbar_{(\theta_{1},\theta_{2})}) \\[5pt]
 &\geq& 2\, \dim_\C H^{\bullet}(A^{\bullet}(M)_{\C};\, \de_{\phi}) \;.
\end{eqnarray*}

But in general, when $\theta_{1}\not=0$, one does not have a double complex structure on $\left( A^{\bullet}(M)_{\C},\, \del_{(\theta_{1},\theta_{2})},\, \delbar_{(\theta_{1},\theta_{2})} \right)$.
In fact, Example \ref{ex:nakamura} shows that the inequality
$$ \dim_\C H^{\bullet}\left(A^{\bullet}(M)_{\C};\, \delbar_{(\theta_{1},\theta_{2})}\right) \geq\dim_\C H^{\bullet}(A^{\bullet}(M)_{\C};\,\de_{\phi})
$$
may fail.
\end{rem}

\subsection{Hodge theory and cohomologies with twisted differentials}

Take a Hermitian metric $g$ on $\left( M ,\, J \right)$.
We consider the adjoint operators $\de_{\phi}^{\ast}$, $\del_{(\theta_{1},\theta_{2})}^{\ast}$, and $\delbar_{(\theta_{1},\theta_{2})}^{\ast}$ of the operators $\de_\phi$, $\del_{(\theta_{1},\theta_{2})}$, and $\delbar_{(\theta_{1},\theta_{2})}$, respectively, with respect to the inner product $\left(\sspace, \ssspace\right)$ induced by $g$.
We define the Laplacian operators
\begin{eqnarray*}
 \Delta_{\de_{\phi}} &:=& \left[ \de_{\phi},\, \de_{\phi}^{\ast} \right] \;:=\; \de_{\phi}\de_{\phi}^{\ast}+\de_{\phi}^{\ast}\de_{\phi} \;, \\[5pt]
 \Delta_{\del_{(\theta_{1},\theta_{2})}} &:=& \left[ \del_{(\theta_{1},\theta_{2})},\, \del_{(\theta_{1},\theta_{2})}^{\ast} \right] \;:=\; \del_{(\theta_{1},\theta_{2})}\del_{(\theta_{1},\theta_{2})}^{\ast}+\del_{(\theta_{1},\theta_{2})}^{\ast}\del_{(\theta_{1},\theta_{2})} \;, \\[5pt]
 \Delta_{\delbar_{(\theta_{1},\theta_{2})}} &:=& \left[ \delbar_{(\theta_{1},\theta_{2})} ,\, \delbar_{(\theta_{1},\theta_{2})}^{\ast} \right] \;:=\; \delbar_{(\theta_{1},\theta_{2})}\delbar_{(\theta_{1},\theta_{2})}^{\ast}+\delbar_{(\theta_{1},\theta_{2})}^{\ast}\delbar_{(\theta_{1},\theta_{2})} \,, \\[5pt]
 \Delta_{BC,\del_{(\theta_{1},\theta_{2})}, \delbar_{(\theta_{1},\theta_{2})}} &:=& \left(\del_{(\theta_{1},\theta_{2})}\delbar_{(\theta_{1},\theta_{2})}\right)\left(\del_{(\theta_{1},\theta_{2})}\delbar_{(\theta_{1},\theta_{2})}\right)^\ast+\left(\del_{(\theta_{1},\theta_{2})}\delbar_{(\theta_{1},\theta_{2})}\right)^\ast\left(\del_{(\theta_{1},\theta_{2})}\delbar_{(\theta_{1},\theta_{2})}\right) \\[5pt]
 && + \left(\delbar_{(\theta_{1},\theta_{2})}^\ast\del_{(\theta_{1},\theta_{2})}\right)\left(\delbar_{(\theta_{1},\theta_{2})}^\ast\del_{(\theta_{1},\theta_{2})}\right)^\ast + \left(\delbar_{(\theta_{1},\theta_{2})}^\ast\del_{(\theta_{1},\theta_{2})}\right)^\ast\left(\delbar_{(\theta_{1},\theta_{2})}^\ast\del_{(\theta_{1},\theta_{2})}\right) \\[5pt]
 && + \delbar_{(\theta_{1},\theta_{2})}^\ast\delbar_{(\theta_{1},\theta_{2})} + \del_{(\theta_{1},\theta_{2})}^\ast\del_{(\theta_{1},\theta_{2})} \;, \\[5pt]
 \Delta_{A,\del_{(\theta_{1},\theta_{2})}, \delbar_{(\theta_{1},\theta_{2})}} &:=& \del_{(\theta_{1},\theta_{2})}\del_{(\theta_{1},\theta_{2})}^\ast + \delbar_{(\theta_{1},\theta_{2})}\delbar_{(\theta_{1},\theta_{2})}^\ast \\[5pt]
 && + \left(\del_{(\theta_{1},\theta_{2})}\delbar_{(\theta_{1},\theta_{2})}\right)^\ast\left(\del_{(\theta_{1},\theta_{2})}\delbar_{(\theta_{1},\theta_{2})}\right) + \left(\del_{(\theta_{1},\theta_{2})}\delbar_{(\theta_{1},\theta_{2})}\right)\left(\del_{(\theta_{1},\theta_{2})}\delbar_{(\theta_{1},\theta_{2})}\right)^\ast \\[5pt]
 && + \left(\delbar_{(\theta_{1},\theta_{2})}\del_{(\theta_{1},\theta_{2})}^\ast\right)^\ast\left(\delbar_{(\theta_{1},\theta_{2})}\del_{(\theta_{1},\theta_{2})}^\ast\right)+\left(\delbar_{(\theta_{1},\theta_{2})}\del_{(\theta_{1},\theta_{2})}^\ast\right)\left(\delbar_{(\theta_{1},\theta_{2})}\del_{(\theta_{1},\theta_{2})}^\ast\right)^\ast \;.
\end{eqnarray*}

Note that the principal parts of the above operators are equal to the principal parts of the corresponding operators with $\phi=0$ and $\left(\theta_1,\theta_2\right)=(0,0)$. In particular, $\Delta_{\de_{\phi}}$, and $\Delta_{\del_{(\theta_{1},\theta_{2})}}$, and $\Delta_{\delbar_{(\theta_{1},\theta_{2})}}$ are $2$nd order self-adjoint elliptic differential operators, see \cite[page 22]{Sim}. In particular, one has the orthogonal decompositions
\begin{eqnarray*}
 A^{\bullet}(M)_{\C} &=& \ker\Delta_{\de_{\phi}} \stackrel{\perp}{\oplus} \imm\Delta_{\de_{\phi}} \;, \\[5pt]
 A^{\bullet}(M)_{\C} &=& \ker\Delta_{\del_{(\theta_{1},\theta_{2})}} \stackrel{\perp}{\oplus} \imm\Delta_{\del_{(\theta_{1},\theta_{2})}} \;, \\[5pt]
 A^{\bullet}(M)_{\C} &=& \ker\Delta_{\del_{(\theta_{1},\theta_{2})}} \stackrel{\perp}{\oplus} \imm\Delta_{\del_{(\theta_{1},\theta_{2})}} \;, \\[5pt]
\end{eqnarray*}
with respect to the inner product induced by $g$, and hence the isomorphisms
\begin{eqnarray*}
 H^{\bullet}\left(A^{\bullet}(M)_{\C};\, \de_{\phi}\right) &\simeq& \ker\Delta_{\de_{\phi}} \;, \\[5pt]
 H^{\bullet}\left(A^{\bullet}(M)_{\C};\, \del_{(\theta_{1},\theta_{2})}\right) &\simeq& \ker\Delta_{\del_{(\theta_{1},\theta_{2})}} \;, \\[5pt]
 H^{\bullet}\left(A^{\bullet}(M)_{\C};\, \delbar_{(\theta_{1},\theta_{2})}\right) &\simeq& \ker\Delta_{\delbar_{(\theta_{1},\theta_{2})}}
\end{eqnarray*}
of vector spaces, depending on the metric, see \cite[page 22]{Sim}.

Furthermore, M. Schweitzer proved that $\tilde\Delta_{BC}:=\Delta_{BC,\del_{(0,0)}, \delbar_{(0,0)}}$ and $\tilde\Delta_{A}:=\Delta_{A,\del_{(0,0)}, \delbar_{(0,0)}}$ are $4$th order self-adjoint elliptic differential operators, in \cite[\S2.b, \S2.c]{schweitzer}, see also \cite[Proposition 5]{kodaira-spencer-3}. Hence we have the following result.

\begin{thm}
 Let $\left( M ,\, J \right)$ be a compact complex manifold endowed with a Hermitian metric $g$. 
 Take $\theta_{1}, \theta_{2} \in H^{1,0}_{BC}(M)$.
 Then the operators $\Delta_{BC,\del_{(\theta_{1},\theta_{2})}, \delbar_{(\theta_{1},\theta_{2})}}$ and $\Delta_{A,\del_{(\theta_{1},\theta_{2})}, \delbar_{(\theta_{1},\theta_{2})}}$ are $4$th order self-adjoint elliptic differential operators.
\end{thm}

For the classical theory of self-adjoint elliptic differential operators, see, e.g., \cite[page 450]{kodaira}, we get the following corollaries.

\begin{cor}\label{cor:bc-isom-ker}
 Let $\left( M ,\, J \right)$ be a compact complex manifold endowed with a Hermitian metric $g$.
 Take $\theta_{1}, \theta_{2} \in H^{1,0}_{BC}(M)$.
 \begin{itemize}
  \item There is an orthogonal decomposition
        $$ A^{\bullet}(M)_\C \;=\; \ker\Delta_{BC,\del_{(\theta_{1},\theta_{2})}, \delbar_{(\theta_{1},\theta_{2})}} \stackrel{\perp}{\oplus} \imm\Delta_{BC,\del_{(\theta_{1},\theta_{2})}, \delbar_{(\theta_{1},\theta_{2})}} \;, $$
        with respect to the inner product induced by $g$.
        Hence there is an isomorphism
        $$ H^{\bullet}\left(A^{\bullet}(M)_{\C};\, \del_{(\theta_{1},\theta_{2})},\, \delbar_{(\theta_{1},\theta_{2})};\, \del_{(\theta_{1},\theta_{2})}\delbar_{(\theta_{1},\theta_{2})}\right) \;=\; \ker\Delta_{BC,\del_{(\theta_{1},\theta_{2})}, \delbar_{(\theta_{1},\theta_{2})}} \;, $$
        depending on the metric.
  \item There is an orthogonal decomposition
        $$ A^{\bullet}(M)_\C \;=\; \ker\Delta_{A,\del_{(\theta_{1},\theta_{2})}, \delbar_{(\theta_{1},\theta_{2})}} \stackrel{\perp}{\oplus} \imm\Delta_{A,\del_{(\theta_{1},\theta_{2})}, \delbar_{(\theta_{1},\theta_{2})}} \;, $$
        with respect to the inner product induced by $g$.
        Hence there is an isomorphism
        $$ H^{\bullet}\left(A^{\bullet}(M)_{\C};\, \del_{(\theta_{1},\theta_{2})}\delbar_{(\theta_{1},\theta_{2})};\, \del_{(\theta_{1},\theta_{2})},\, \delbar_{(\theta_{1},\theta_{2})}\right) \;=\; \ker\Delta_{A,\del_{(\theta_{1},\theta_{2})}, \delbar_{(\theta_{1},\theta_{2})}} \;, $$
        depending on the metric.
 \end{itemize}
\end{cor}

In particular, it follows that the Hodge-$*$-operator induces isomorphisms between cohomologies. (With abuse of notation, for $\zeta_{1},\zeta_{2}\in H^{0,1}_{BC}(X)$, denote $\del_{(\zeta_{1},\zeta_{2})} := \partial+L_{\zeta_{2}}+L_{\overline{\zeta_{1}}}$ and $\delbar_{(\zeta_{1},\zeta_{2})} := \delbar -L_{\overline{\zeta_{2}}}+L_{\zeta_{1}}$.)

\begin{cor}\label{cor:star-duality}
 Let $\left( M ,\, J \right)$ be a compact complex manifold, of complex dimension $n$, endowed with a Hermitian metric $g$.
 Take $\theta_{1}, \theta_{2} \in H^{1,0}_{BC}(M)$, and $\phi:=\theta_{1}+\overline{\theta_{1}}+\theta_{2}-\overline{\theta_{2}}$.
 Fix a $J$-Hermitian metric $g$, and consider the associated $\C$-anti-linear Hodge-$*$-operator $\overline{\ast}\colon A^{\bullet}(M)_{\C}\to A^{2n-\bullet}(M)_{\C}$. It induces the isomorphisms
 \begin{eqnarray*}
  \overline{\ast} \colon \lefteqn{ H^{\bullet}\left(A^\bullet(M)_\C; \de_\phi\right) \stackrel{\simeq}{\to} H^{2n-\bullet}\left(A^\bullet(M)_\C; \de_{-\phi}\right) \;,} \\[5pt]
  \overline{\ast} \colon \lefteqn{ H^{\bullet}\left(A^\bullet(M)_\C; \del_{(\theta_1,\theta_2)}\right) \stackrel{\simeq}{\to} H^{2n-\bullet}\left(A^\bullet(M)_\C; \del_{(-\theta_2,-\theta_1)}\right) \;, } \\[5pt]
  \overline{\ast} \colon \lefteqn{ H^{\bullet}\left(A^\bullet(M)_\C; \delbar_{(\theta_1,\theta_2)}\right) \stackrel{\simeq}{\to} H^{2n-\bullet}\left(A^\bullet(M)_\C; \delbar_{(-\theta_2,-\theta_1)}\right) \;, } \\[5pt]
  \overline{\ast} \colon \lefteqn{ H^{\bullet}\left(A^\bullet(M)_\C; \del_{(\theta_1,\theta_2)},\, \delbar_{(\theta_1,\theta_2)};\, \del_{(\theta_1,\theta_2)}\delbar_{(\theta_1,\theta_2)}\right) } \\[5pt]
  && \stackrel{\simeq}{\to} H^{2n-\bullet}\left(A^\bullet(M)_\C; \del_{(-\theta_2,-\theta_1)}\delbar_{(-\theta_2,-\theta_1)};\, \del_{(-\theta_2,-\theta_1)},\, \delbar_{(-\theta_2,-\theta_1)} \right) \;.
 \end{eqnarray*}
\end{cor}

\begin{proof}
 Note that
 \begin{eqnarray*}
  \overline{\ast} \, \Delta_{\de_{\phi}} &=& \Delta_{\de_{-\phi}} \, \overline{\ast} \;, \\[5pt]
  \overline{\ast} \, \Delta_{\del_{(\theta_1,\theta_2)}} &=& \Delta_{\del_{(-\theta_2,-\theta_1)}} \, \overline{\ast} \;, \\[5pt]
  \overline{\ast} \, \Delta_{\delbar_{(\theta_1,\theta_2)}} &=& \Delta_{\delbar_{(-\theta_2,-\theta_1)}} \, \overline{\ast} \;, \\[5pt]
  \overline{\ast} \, \Delta_{BC_{(\theta_1,\theta_2)}} &=& \Delta_{A_{(-\theta_2,-\theta_1)}} \, \overline{\ast} \;.
 \end{eqnarray*}
 The statement follows from \cite[page 22]{Sim} and Corollary \ref{cor:bc-isom-ker}.
\end{proof}

\subsection{Hodge theory on K\"ahler manifolds with twisted differentials}

Consider the case of a compact complex manifold endowed with a K\"ahler metric. Thanks to the K\"ahler identities for twisted differentials, we have the following analogue of the classical Hodge decomposition theorem.

\begin{prop}\label{prop:kahler-lapl}
 Let $\left( M ,\, J \right)$ be a compact complex manifold endowed with a K\"ahler metric $g$.
 Take $[\phi]\in H^1(M;\C)$.
 Consider $\theta_{1}, \theta_{2} \in H^{1,0}_{BC}(M)$ such that $\phi = \theta_1+\overline\theta_1 + \theta_2-\overline\theta_2$.
 Then
 $$ \Delta_{\de_{\phi}} \;=\; 2\, \Delta_{\del_{(\theta_{1},\theta_{2})}} \;=\; 2\, \Delta_{\delbar_{(\theta_{1},\theta_{2})}} $$
 and
 $$ \Delta_{BC,\del_{(\theta_{1},\theta_{2})}, \delbar_{(\theta_{1},\theta_{2})}} \;=\; \Delta_{\delbar_{(\theta_{1},\theta_{2})}}^2 + \del_{(\theta_{1},\theta_{2})}^\ast\del_{(\theta_{1},\theta_{2})}+\delbar_{(\theta_{1},\theta_{2})}^\ast\delbar_{(\theta_{1},\theta_{2})} $$
 and
 $$ \Delta_{A,\del_{(\theta_{1},\theta_{2})}, \delbar_{(\theta_{1},\theta_{2})}} \;=\; \Delta_{\delbar_{(\theta_{1},\theta_{2})}}^2 + \del_{(\theta_{1},\theta_{2})}\del_{(\theta_{1},\theta_{2})}^\ast+\delbar_{(\theta_{1},\theta_{2})}\delbar_{(\theta_{1},\theta_{2})}^\ast \;. $$
\end{prop}

\begin{proof}
 For the sake of completeness, we detail the proof.

 Take $[\phi]\in H^1(M;\C)$. Then $[\phi]=[\Re\phi]+\sqrt{-1}\,[\Im\phi]$ where $[\Re\phi] \in H^1(M;\R)\subset H^1(M;\C)$ and $[\Im\phi] \in H^1(M;\R)\subset H^1(M;\C)$. By the Hodge decomposition theorem for compact K\"ahler manifolds, one has that the identity induces the isomorphism $H^{1,0}_{BC}(M) \oplus H^{0,1}_{BC}(M) \stackrel{\simeq}{\to} H^1(M;\C)$. Hence there exists $\theta_1 \in H^{1,0}_{BC}(M)$ such that $\Re\phi = \theta_1 + \overline\theta_1$, and there exists $\theta_2 \in H^{1,0}_{BC}(M)$ such that $\sqrt{-1}\Im\Phi = \theta_2 - \overline\theta_2$.

 Note that
 $$ \de_\phi \;=\; \de + L_\phi \;=\; \del + L_{\theta_2} + L_{\overline\theta_1} + \delbar - L_{\overline\theta_2} + L_{\theta_1} \;=\; \del_{\left(\theta_1,\theta_2\right)} + \delbar_{\left(\theta_1,\theta_2\right)} \;. $$

 Firstly, note that, by the K\"ahler identities for twisted differentials, we have:
 \begin{eqnarray*}
  \left[ \del_{\left(\theta_1,\theta_2\right)},\, \delbar_{\left(\theta_1,\theta_2\right)}^{\ast} \right]
  &=& -\sqrt{-1}\, \left[ \del_{\left(\theta_1,\theta_2\right)},\, \Lambda_\omega \, \del_{\left(\theta_1,\theta_2\right)} - \del_{\left(\theta_1,\theta_2\right)} \, \Lambda_\omega \right] \\[5pt]
  &=& -\sqrt{-1}\, \left( \del_{\left(\theta_1,\theta_2\right)} \, \Lambda_\omega \, \del_{\left(\theta_1,\theta_2\right)} - \del_{\left(\theta_1,\theta_2\right)}^2 \, \Lambda_\omega + \Lambda_\omega \, \del_{\left(\theta_1,\theta_2\right)}^2 - \del_{\left(\theta_1,\theta_2\right)} \, \Lambda_\omega \, \del_{\left(\theta_1,\theta_2\right)} \right) \\[5pt]
  &=& 0
 \end{eqnarray*}
 and, by conjugation,
 $$ \left[ \delbar_{\left(\theta_1,\theta_2\right)},\, \del_{\left(\theta_1,\theta_2\right)}^{\ast} \right] \;=\; 0 \;, $$
 where $\omega$ denotes the $(1,1)$-form associated to $g$.

 Therefore
 \begin{eqnarray*}
  \Delta_{\de_{\phi}} &=& \left[ \de_\phi,\, \de_\phi^\ast \right] \;=\; \left[ \del_{\left(\theta_1,\theta_2\right)} + \delbar_{\left(\theta_1,\theta_2\right)} ,\, \del_{\left(\theta_1,\theta_2\right)}^{\ast} + \delbar_{\left(\theta_1,\theta_2\right)}^{\ast} \right] \\[5pt]
  &=& \left[ \del_{\left(\theta_1,\theta_2\right)} ,\, \del_{\left(\theta_1,\theta_2\right)}^{\ast} \right] + \left[ \del_{\left(\theta_1,\theta_2\right)} ,\, \delbar_{\left(\theta_1,\theta_2\right)}^{\ast} \right] + \left[ \delbar_{\left(\theta_1,\theta_2\right)} ,\, \del_{\left(\theta_1,\theta_2\right)}^{\ast} \right] + \left[ \delbar_{\left(\theta_1,\theta_2\right)} ,\, \delbar_{\left(\theta_1,\theta_2\right)}^{\ast} \right] \\[5pt]
  &=& \Delta_{\del_{(\theta_{1},\theta_{2})}} + \Delta_{\delbar_{(\theta_{1},\theta_{2})}} \;.
 \end{eqnarray*}
 Hence we have to show that
 $$ \Delta_{\del_{(\theta_{1},\theta_{2})}} \;=\; \Delta_{\delbar_{(\theta_{1},\theta_{2})}} \;. $$
 Indeed, by using the K\"ahler identities, we have
 \begin{eqnarray*}
  \Delta_{\del_{(\theta_{1},\theta_{2})}} &=& \left[ \del_{(\theta_{1},\theta_{2})},\, \del_{(\theta_{1},\theta_{2})}^\ast \right] \;=\; \sqrt{-1}\, \left[ \del_{(\theta_{1},\theta_{2})},\, \left[ \Lambda_{\omega} ,\, \delbar_{(\theta_{1},\theta_{2})} \right] \right] \\[5pt]
  &=& \sqrt{-1}\, \left[ \Lambda_{\omega} ,\, \left[ \del_{(\theta_{1},\theta_{2})},\, \delbar_{(\theta_{1},\theta_{2})} \right] \right] + \sqrt{-1}\, \left[ \delbar_{(\theta_{1},\theta_{2})} ,\, \left[ \del_{(\theta_{1},\theta_{2})},\, \Lambda_{\omega} \right] \right] \\[5pt]
  &=& \left[ \delbar_{(\theta_{1},\theta_{2})},\, \delbar_{(\theta_{1},\theta_{2})}^\ast \right] \;=\; \Delta_{\delbar_{(\theta_{1},\theta_{2})}} \;.
 \end{eqnarray*}

 Again by the K\"ahler identities, we have (compare \cite[Proposition 6]{kodaira-spencer-3}, \cite[Proposition 2.4]{schweitzer})
 \begin{eqnarray*}
  \Delta_{\delbar_{(\theta_{1},\theta_{2})}}^2 &=& \Delta_{\delbar_{(\theta_{1},\theta_{2})}} \, \Delta_{\del_{(\theta_{1},\theta_{2})}} \\[5pt]
  &=& \delbar_{(\theta_{1},\theta_{2})}\delbar_{(\theta_{1},\theta_{2})}^\ast\del_{(\theta_{1},\theta_{2})}\del_{(\theta_{1},\theta_{2})}^\ast + \delbar_{(\theta_{1},\theta_{2})}^\ast\delbar_{(\theta_{1},\theta_{2})}\del_{(\theta_{1},\theta_{2})}\del_{(\theta_{1},\theta_{2})}^\ast \\[5pt]
  && + \delbar_{(\theta_{1},\theta_{2})}\delbar_{(\theta_{1},\theta_{2})}^\ast\del_{(\theta_{1},\theta_{2})}^\ast\del_{(\theta_{1},\theta_{2})} + \delbar_{(\theta_{1},\theta_{2})}^\ast\delbar_{(\theta_{1},\theta_{2})}\del_{(\theta_{1},\theta_{2})}^\ast\del_{(\theta_{1},\theta_{2})} \\[5pt]
  &=& - \delbar_{(\theta_{1},\theta_{2})}\del_{(\theta_{1},\theta_{2})}\delbar_{(\theta_{1},\theta_{2})}^\ast\del_{(\theta_{1},\theta_{2})}^\ast - \delbar_{(\theta_{1},\theta_{2})}^\ast\del_{(\theta_{1},\theta_{2})}\delbar_{(\theta_{1},\theta_{2})}\del_{(\theta_{1},\theta_{2})}^\ast \\[5pt]
  && - \delbar_{(\theta_{1},\theta_{2})}\del_{(\theta_{1},\theta_{2})}^\ast\delbar_{(\theta_{1},\theta_{2})}^\ast\del_{(\theta_{1},\theta_{2})} - \delbar_{(\theta_{1},\theta_{2})}^\ast\del_{(\theta_{1},\theta_{2})}^\ast\delbar_{(\theta_{1},\theta_{2})}\del_{(\theta_{1},\theta_{2})} \\[5pt]
  &=& \del_{(\theta_{1},\theta_{2})}\delbar_{(\theta_{1},\theta_{2})}\delbar_{(\theta_{1},\theta_{2})}^\ast\del_{(\theta_{1},\theta_{2})}^\ast + \delbar_{(\theta_{1},\theta_{2})}^\ast\del_{(\theta_{1},\theta_{2})}\del_{(\theta_{1},\theta_{2})}^\ast\delbar_{(\theta_{1},\theta_{2})} \\[5pt]
  && + \del_{(\theta_{1},\theta_{2})}^\ast\delbar_{(\theta_{1},\theta_{2})}\delbar_{(\theta_{1},\theta_{2})}^\ast\del_{(\theta_{1},\theta_{2})} + \delbar_{(\theta_{1},\theta_{2})}^\ast\del_{(\theta_{1},\theta_{2})}^\ast\del_{(\theta_{1},\theta_{2})}\delbar_{(\theta_{1},\theta_{2})} \\[5pt]
  &=& \Delta_{BC,\del_{(\theta_{1},\theta_{2})}, \delbar_{(\theta_{1},\theta_{2})}} - \delbar_{(\theta_{1},\theta_{2})}^\ast\delbar_{(\theta_{1},\theta_{2})} - \del_{(\theta_{1},\theta_{2})}^\ast\del_{(\theta_{1},\theta_{2})} \;.
 \end{eqnarray*}
 Analogously,
 \begin{eqnarray*}
  \Delta_{\delbar_{(\theta_{1},\theta_{2})}}^2 &=& \Delta_{\delbar_{(\theta_{1},\theta_{2})}} \, \Delta_{\del_{(\theta_{1},\theta_{2})}} \\[5pt]
  &=& \delbar_{(\theta_{1},\theta_{2})}\delbar_{(\theta_{1},\theta_{2})}^\ast\del_{(\theta_{1},\theta_{2})}\del_{(\theta_{1},\theta_{2})}^\ast + \delbar_{(\theta_{1},\theta_{2})}^\ast\delbar_{(\theta_{1},\theta_{2})}\del_{(\theta_{1},\theta_{2})}\del_{(\theta_{1},\theta_{2})}^\ast \\[5pt]
  && + \delbar_{(\theta_{1},\theta_{2})}\delbar_{(\theta_{1},\theta_{2})}^\ast\del_{(\theta_{1},\theta_{2})}^\ast\del_{(\theta_{1},\theta_{2})} + \delbar_{(\theta_{1},\theta_{2})}^\ast\delbar_{(\theta_{1},\theta_{2})}\del_{(\theta_{1},\theta_{2})}^\ast\del_{(\theta_{1},\theta_{2})} \\[5pt]
  &=& - \delbar_{(\theta_{1},\theta_{2})}\del_{(\theta_{1},\theta_{2})}\delbar_{(\theta_{1},\theta_{2})}^\ast\del_{(\theta_{1},\theta_{2})}^\ast - \delbar_{(\theta_{1},\theta_{2})}^\ast\del_{(\theta_{1},\theta_{2})}\delbar_{(\theta_{1},\theta_{2})}\del_{(\theta_{1},\theta_{2})}^\ast \\[5pt]
  && - \delbar_{(\theta_{1},\theta_{2})}\del_{(\theta_{1},\theta_{2})}^\ast\delbar_{(\theta_{1},\theta_{2})}^\ast\del_{(\theta_{1},\theta_{2})} - \delbar_{(\theta_{1},\theta_{2})}^\ast\del_{(\theta_{1},\theta_{2})}^\ast\delbar_{(\theta_{1},\theta_{2})}\del_{(\theta_{1},\theta_{2})} \\[5pt]
  &=& \del_{(\theta_{1},\theta_{2})}\delbar_{(\theta_{1},\theta_{2})}\delbar_{(\theta_{1},\theta_{2})}^\ast\del_{(\theta_{1},\theta_{2})}^\ast + \del_{(\theta_{1},\theta_{2})}\delbar_{(\theta_{1},\theta_{2})}^\ast\delbar_{(\theta_{1},\theta_{2})}\del_{(\theta_{1},\theta_{2})}^\ast \\[5pt]
  && + \delbar_{(\theta_{1},\theta_{2})}\del_{(\theta_{1},\theta_{2})}^\ast\del_{(\theta_{1},\theta_{2})}\delbar_{(\theta_{1},\theta_{2})}^\ast + \delbar_{(\theta_{1},\theta_{2})}^\ast\del_{(\theta_{1},\theta_{2})}^\ast\del_{(\theta_{1},\theta_{2})}\delbar_{(\theta_{1},\theta_{2})} \\[5pt]
  &=& \Delta_{A,\del_{(\theta_{1},\theta_{2})}, \delbar_{(\theta_{1},\theta_{2})}} - \delbar_{(\theta_{1},\theta_{2})}\delbar_{(\theta_{1},\theta_{2})}^\ast - \del_{(\theta_{1},\theta_{2})}\del_{(\theta_{1},\theta_{2})}^\ast \;.
 \end{eqnarray*}

 These identities conclude the proof.
\end{proof}

In particular, it follows that, on compact K\"ahler manifolds, all the above cohomologies are isomorphic.

\begin{cor}\label{cor:hodge-dec}
 Let $\left( M ,\, J \right)$ be a compact complex manifold endowed with a K\"ahler metric $g$.
 Take $[\phi]\in H^1(M;\C)$.
 Consider $\theta_{1}, \theta_{2} \in H^{1,0}_{BC}(M)$ such that $\phi = \theta_1+\overline\theta_1 + \theta_2-\overline\theta_2$.
 Then there are isomorphisms
 $$ \xymatrix{
 & H^{\bullet}\left(A^{\bullet}(M)_{\C};\, \del_{(\theta_{1},\theta_{2})},\, \delbar_{(\theta_{1},\theta_{2})};\, \del_{(\theta_{1},\theta_{2})}\delbar_{(\theta_{1},\theta_{2})}\right) \ar[ld]_{\simeq} \ar[d]_{\simeq} \ar[rd]^{\simeq} & \\
 H^{\bullet}\left(A^{\bullet}(M)_{\C};\, \del_{(\theta_{1},\theta_{2})}\right) \ar[rd]_{\simeq} & H^{\bullet}\left(A^{\bullet}(M)_{\C};\, \de_{\phi}\right) \ar[d]^{\simeq} & H^{\bullet}\left(A^{\bullet}(M)_{\C};\, \delbar_{(\theta_{1},\theta_{2})}\right) \ar[ld]^{\simeq} \\
 & H^{\bullet}\left(A^{\bullet}(M)_{\C};\, \del_{(\theta_{1},\theta_{2})}\delbar_{(\theta_{1},\theta_{2})};\, \del_{(\theta_{1},\theta_{2})},\, \delbar_{(\theta_{1},\theta_{2})}\right) &
 } $$
 of $\Z$-graded vector spaces.
\end{cor}

Since the isomorphisms in \cite[page 22]{Sim} and Corollary \ref{cor:bc-isom-ker} depend on the K\"ahler metric, also the isomorphisms in Corollary \ref{cor:hodge-dec}, {\itshape a priori}, depend on the K\"ahler metric. In fact, the following result holds, analogous to the $\del\delbar$-Lemma.

\begin{thm}\label{thm:kahler-deldelbar-tw}
 Let $\left( M ,\, J \right)$ be a compact complex manifold endowed with a K\"ahler metric $g$.
 Take $\theta_{1}, \theta_{2} \in H^{1,0}_{BC}(M)$.
 Then the natural map
 \begin{eqnarray*}
  \lefteqn{H^{\bullet}\left(A^{\bullet}(M)_{\C};\, \del_{(\theta_{1},\theta_{2})},\, \delbar_{(\theta_{1},\theta_{2})};\, \del_{(\theta_{1},\theta_{2})}\delbar_{(\theta_{1},\theta_{2})}\right)} \\[5pt]
  && \to H^{\bullet}\left(A^{\bullet}(M)_{\C};\, \del_{(\theta_{1},\theta_{2})}\delbar_{(\theta_{1},\theta_{2})};\, \del_{(\theta_{1},\theta_{2})},\, \delbar_{(\theta_{1},\theta_{2})}\right)
 \end{eqnarray*}
 induced by the identity is injective.
\end{thm}

\begin{proof}
 We detail the proof, which follows the argument in \cite[pages 266--267]{deligne-griffiths-morgan-sullivan}.

 We prove that
 $$ \ker \del_{(\theta_{1},\theta_{2})} \cap \ker \delbar_{(\theta_{1},\theta_{2})} \cap \left( \imm \del_{(\theta_1,\theta_2)} + \imm \delbar_{(\theta_1,\theta_2)} \right) \;=\; \imm \del_{(\theta_{1},\theta_{2})}\delbar_{(\theta_{1},\theta_{2})} \;. $$
 Consider $\alpha = \del_{(\theta_1,\theta_2)}\beta+\delbar_{(\theta_1,\theta_2)}\gamma \in A^k(M)_\C$ such that $\del_{(\theta_{1},\theta_{2})}\alpha = \delbar_{(\theta_{1},\theta_{2})}\alpha = 0$, where $\beta\in A^{k-1}(M)_\C$ and $\gamma\in A^{k-1}(M)_\C$.
 
 Fix a Hermitian metric $g$.
 By Corollary \ref{cor:bc-isom-ker} and Proposition \ref{prop:kahler-lapl}, one has
 \begin{eqnarray*}
  \alpha &\in& \imm\del_{(\theta_1,\theta_2)}+\imm\delbar_{(\theta_1,\theta_2)} \;\subseteq\; \imm\Delta_{A,\del_{(\theta_1,\theta_2)},\delbar_{(\theta_1,\theta_2)}} \\[5pt]
  &\perp& \ker\Delta_{A,\del_{(\theta_1,\theta_2)},\delbar_{(\theta_1,\theta_2)}} \;=\; \ker\Delta_{BC,\del_{(\theta_1,\theta_2)},\delbar_{(\theta_1,\theta_2)}} \;,
 \end{eqnarray*}
 therefore, again by Corollary \ref{cor:bc-isom-ker}, one has
 $$ \alpha \;\in\; \imm\Delta_{BC,\del_{(\theta_1,\theta_2)},\delbar_{(\theta_1,\theta_2)}} \;=\; \imm\del_{(\theta_1,\theta_2)}\delbar_{(\theta_1,\theta_2)} \stackrel{\perp}{\oplus} \left( \imm\del_{(\theta_1,\theta_2)}^\ast + \imm\delbar_{(\theta_1,\theta_2)}^\ast \right) \;. $$
  
 Since $\del_{(\theta_{1},\theta_{2})}\alpha=0$, then $\alpha\perp\imm\del_{(\theta_{1},\theta_{2})}^\ast$. Since $\delbar_{(\theta_{1},\theta_{2})}\alpha=0$, then $\alpha\perp\imm\delbar_{(\theta_{1},\theta_{2})}^\ast$. Hence $\alpha\in\imm\del_{(\theta_{1},\theta_{2})}\delbar_{(\theta_{1},\theta_{2})}$. This concludes the proof.
\end{proof}

\begin{cor}\label{thm:kahler-hodge-tw}
 Let $\left( M ,\, J \right)$ be a compact complex manifold endowed with a K\"ahler metric.
 Take $[\phi]\in H^1(M;\C)$.
 Consider $\theta_{1}, \theta_{2} \in H^{1,0}_{BC}(M)$ such that $\phi = \theta_1+\overline\theta_1 + \theta_2-\overline\theta_2$.
 Then the natural maps
 $$ \xymatrix{
 & H^{\bullet}\left(A^{\bullet}(M)_{\C};\, \del_{(\theta_{1},\theta_{2})},\, \delbar_{(\theta_{1},\theta_{2})};\, \del_{(\theta_{1},\theta_{2})}\delbar_{(\theta_{1},\theta_{2})}\right) \ar[ld]_{\iota_{BC,\del}} \ar[d]_{\iota_{BC,\de}} \ar[rd]^{\iota_{BC,\delbar}} \ar@/_5pc/[dd]_{\iota_{BC,A}} & \\
 H^{\bullet}\left(A^{\bullet}(M)_{\C};\, \del_{(\theta_{1},\theta_{2})}\right) \ar[rd]_{\iota_{\del,A}} & H^{\bullet}\left(A^{\bullet}(M)_{\C};\, \de_{\phi}\right) \ar[d]^{\iota_{\de,A}} & H^{\bullet}\left(A^{\bullet}(M)_{\C};\, \delbar_{(\theta_{1},\theta_{2})}\right) \ar[ld]^{\iota_{\delbar,A}} \\
 & H^{\bullet}\left(A^{\bullet}(M)_{\C};\, \del_{(\theta_{1},\theta_{2})}\delbar_{(\theta_{1},\theta_{2})};\, \del_{(\theta_{1},\theta_{2})},\, \delbar_{(\theta_{1},\theta_{2})}\right) &
 } $$
 induced by the identity are isomorphisms.
\end{cor}

\begin{proof}
 By \cite[Lemma 5.15]{deligne-griffiths-morgan-sullivan}, see also \cite[Lemma 1.4]{angella-tomassini-5}, the maps $\iota_{BC,A}$, $\iota_{BC,\del}$, $\iota_{BC,\delbar}$, and $\iota_{BC,\de}$ are injective, and the maps $\iota_{BC,A}$, $\iota_{\del,A}$, $\iota_{\delbar,A}$, and $\iota_{\de,A}$ are surjective. 
 By Corollary \ref{cor:hodge-dec}, they are in fact isomorphisms, being either injective or surjective maps between finite-dimensional vector spaces of the same dimension.
\end{proof}

\subsection{Homologies with twisted differentials}

Consider the space $D^{\bullet}(M)_{\C}=D^{\bullet,\bullet}(M)$ of currents, where we denote by $D^{p}(M)_{\C}$ the space of complex $(\dim_{\R}M-p)$-dimensional currents.
Then we also have the bi-differential ($\Z$-graded) algebra
\[ \left( D^{\bullet}(M)_{\C},\, \del_{(\theta_{1},\theta_{2})},\, \delbar_{(\theta_{1},\theta_{2})} \right) \;. \]
We consider the inclusion
\[ T\colon A^{\bullet}(M)_{\C} \to D^{\bullet}(M)_{\C} \;, \qquad T_{\eta} \;:=\; \int_{M}\eta\wedge \sspace \;. \]

Then we have the following result.
\begin{thm}
 Let $\left( M ,\, J \right)$ be a compact complex manifold endowed with a Hermitian metric $g$. 
 Take $\theta_{1}, \theta_{2} \in H^{1,0}_{BC}(M)$, and $\phi:=\theta_{1}+\overline{\theta_{1}}+\theta_{2}-\overline{\theta_{2}}$.
 The inclusion $T \colon A^{\bullet}(M)_{\C}\to D^{\bullet}(M)_{\C}$ induces the cohomology  isomorphisms
\begin{eqnarray*}
 H^{\bullet}\left(A^{\bullet}(M)_{\C};\, \de_{\phi}\right) & \stackrel{\simeq}{\to} & H^{\bullet}\left(D^{\bullet}(M)_{\C};\, \de_{\phi}\right) \;, \\[5pt]
 H^{\bullet}\left(A^{\bullet}(M)_{\C};\, \del_{(\theta_{1},\theta_{2})}\right) & \stackrel{\simeq}{\to} & H^{\bullet}\left(D^{\bullet}(M)_{\C};\, \del_{(\theta_{1},\theta_{2})}\right) \;, \\[5pt]
 H^{\bullet}\left(A^{\bullet}(M)_{\C};\, \delbar_{(\theta_{1},\theta_{2})}\right) & \stackrel{\simeq}{\to} & H^{\bullet}\left(D^{\bullet}(M)_{\C};\, \delbar_{(\theta_{1},\theta_{2})}\right) \;, \\[5pt]
 H^{\bullet}\left(A^{\bullet}(M)_{\C};\, \del_{(\theta_{1},\theta_{2})},\, \delbar_{(\theta_{1},\theta_{2})};\, \del_{(\theta_{1},\theta_{2})}\delbar_{(\theta_{1},\theta_{2})}\right) & \stackrel{\simeq}{\to} & H^{\bullet}\left(D^{\bullet}(M)_{\C};\, \del_{(\theta_{1},\theta_{2})},\, \delbar_{(\theta_{1},\theta_{2})};\, \del_{(\theta_{1},\theta_{2})}\delbar_{(\theta_{1},\theta_{2})}\right) \;, \\[5pt]
 H^{\bullet}\left(A^{\bullet}(M)_{\C};\, \del_{(\theta_{1},\theta_{2})}\delbar_{(\theta_{1},\theta_{2})};\, \del_{(\theta_{1},\theta_{2})},\, \delbar_{(\theta_{1},\theta_{2})}\right) & \stackrel{\simeq}{\to} & H^{\bullet}\left(D^{\bullet}(M)_{\C};\, \del_{(\theta_{1},\theta_{2})}\delbar_{(\theta_{1},\theta_{2})};\, \del_{(\theta_{1},\theta_{2})},\, \delbar_{(\theta_{1},\theta_{2})}\right) \;. \\[5pt]
\end{eqnarray*}
\end{thm}

\begin{proof}
Consider each case. For a fixed Hermitian metric, consider the corresponding Laplacian operator $\Delta$.
Then, by \cite[Theorem 4.12]{wells-book}, we have the operators
$$ G \colon A^{\bullet}(M)_{\C}\to A^{\bullet}(M)_{\C} \quad \text{ and } \quad H \colon A^{\bullet}(M)_{\C}\to A^{\bullet}(M)_{\C} \;, $$
where $H$ is given by the projection $A^{\bullet}(M)_{\C}\to {\ker} \Delta$ and $G$ is given by the inverse of the restriction of $\Delta $ on $A^{\bullet}(M)_{\C}\cap ({\ker}\, \Delta)^{\perp}$, such that
\[\Delta \circ G+H \;=\; G\circ \Delta+H \;=\;{\id} \;.
\]
Since $\Delta$ is self-adjoint, $G$ and $H$ are also self-adjoint.
We can define the operators $\Delta$, $G$, and $H$ on $D^{\bullet}(M)_{\C}$.
They still satisfy $\Delta \circ G+H = G\circ \Delta+H = {\id}$.
By the regularity of the kernel of  elliptic differential operators in Sobolev spaces, see, e.g., \cite[Theorem 4.8]{wells-book}, we have 
\[{\ker} \Delta\lfloor_{A^{\bullet}(M)_{\C}} \;=\; {\ker} \Delta\lfloor_{D^{\bullet}(M)_{\C}} \;.
\]
Hence we have 
\[D^{\bullet}(M)_{\C}\;=\;{\ker}\Delta\lfloor_{A^{\bullet}(M)_{\C}}+ \Delta\left({D^{\bullet}(M)_{\C}}\right)\;.\]
For $T\in D^{\bullet}(M)_{\C}$ and $h\in {\ker}\Delta\lfloor_{A^{\bullet}(M)_{\C}}$, suppose that $\Delta\lfloor_{D^{\bullet}(M)_{\C}}T=h$.
Then, by \cite[Theorem 4.9]{wells-book}, we have $T\in A^{\bullet}(M)_{\C}$, and hence $h=0$.
Thus we have
\[D^{\bullet}(M)_{\C}\;=\;{\ker}\Delta\lfloor_{A^{\bullet}(M)_{\C}}\oplus\Delta\left({D^{\bullet}(M)_{\C}}\right)\;.\]
This completes the proof.
\end{proof}

\section{Hodge decomposition with twisted differentials and modifications}

Let $f \colon M_{1}\to M_{2}$ be a holomorphic map between compact complex manifolds $M_1$ and $M_2$.
Take $\theta_{1}, \theta_{2} \in H^{1,0}_{BC}(M_{2})$, and $\phi:=\theta_{1}+\overline{\theta_{1}}+\theta_{2}-\overline{\theta_{2}}$.

\subsection{Modifications and cohomologies with twisted differentials}
Consider the pull-back $f^{\ast} \colon A^{\bullet,\bullet}(M_{2})\to A^{\bullet,\bullet}(M_{1})$. We have $f^{\ast} \theta_{1},f^{\ast}\theta_{2}\in H^{1,0}_{BC}(M_{1})$. Since $f^*$ commutes with $\del$ and $\delbar$, then
$$ f^{\ast} \colon \left (A^{\bullet}(M_{2})_{\C},\, \de_{\phi} \right) \to \left( A^{\bullet}(M_{1})_{\C},\, \de_{f^{\ast}\phi} \right) $$
is a morphism of differential $\Z$-graded complexes, and
$$ f^{\ast} \colon \left (A^{\bullet}(M_{2})_{\C},\, \del_{(\theta_{1},\theta_{2})},\, \delbar_{(\theta_{1},\theta_{2})} \right) \to \left( A^{\bullet}(M_{1})_{\C},\, \del_{(f^{\ast}\theta_{1},f^{\ast}\theta_{2})},\, \delbar_{(f^{\ast}\theta_{1},f^{\ast}\theta_{2})} \right) $$
is a morphism of bi-differential $\Z$-graded complexes. In particular, $f$ induces the maps
\begin{eqnarray*}
f^{\ast}_{dR,\phi} &\colon& H^{\bullet}\left(A^{\bullet}(M_2)_{\C};\, \de_{\phi}\right) \to H^{\bullet}\left(A^{\bullet}(M_1)_{\C};\, \de_{f^\ast\phi}\right) \;, \\[5pt]
f^{\ast}_{\delbar,(\theta_1,\theta_2)} &\colon& H^{\bullet}\left(A^{\bullet}(M_2)_{\C};\, \delbar_{(\theta_{1},\theta_{2})}\right) \to H^{\bullet}\left(A^{\bullet}(M_1)_{\C};\, \delbar_{(f^*\theta_{1},f^*\theta_{2})}\right) \;, \\[5pt]
f^{\ast}_{BC,(\theta_1,\theta_2)} &\colon& H^{\bullet}\left(A^{\bullet}(M_2)_{\C};\, \del_{(\theta_{1},\theta_{2})},\, \delbar_{(\theta_{1},\theta_{2})};\, \del_{(\theta_{1},\theta_{2})}\delbar_{(\theta_{1},\theta_{2})}\right) \\[5pt]
&& \to H^{\bullet}\left(A^{\bullet}(M_1)_{\C};\, \del_{(f^\ast\theta_{1},f^\ast\theta_{2})},\, \delbar_{(f^\ast\theta_{1},f^\ast\theta_{2})};\, \del_{(f^\ast\theta_{1},f^\ast\theta_{2})}\delbar_{(f^\ast\theta_{1},f^\ast\theta_{2})}\right) \;, \\[5pt]
f^\ast_{A,(\theta_1,\theta_2)} &\colon& H^{\bullet}\left(A^{\bullet}(M_2)_{\C};\, \del_{(\theta_{1},\theta_{2})}\delbar_{(\theta_{1},\theta_{2})};\, \del_{(\theta_{1},\theta_{2})},\, \delbar_{(\theta_{1},\theta_{2})}\right) \\[5pt]
&& \to H^{\bullet}\left(A^{\bullet}(M_1)_{\C};\, \del_{(f^\ast\theta_{1},f^\ast\theta_{2})}\delbar_{(f^\ast\theta_{1},f^\ast\theta_{2})};\, \del_{(f^\ast\theta_{1},f^\ast\theta_{2})},\, \delbar_{(f^\ast\theta_{1},f^\ast\theta_{2})}\right) \;.
\end{eqnarray*}

\subsection{Modifications and homologies with twisted differentials}
Suppose that $f$ is proper.
Then we have the map $f_{\ast} \colon D^{\bullet,\bullet}(M_{1})\to  D^{\bullet,\bullet}(M_{2})$, called the direct image map, such that $f_{\ast}$ commutes with $\del$ and $\delbar$, and $f_{\ast}\left(f^{\ast}\alpha\wedge C\right) = \alpha\wedge f_\ast C$ for any $\alpha\in A^{\bullet,\bullet}(M_{2})$ and  $C\in D^{\bullet,\bullet}(M_{1})$.
Hence the map
$$ f_{\ast} \colon \left( D^{\bullet}(M_{1})_{\C},\, \de_{f^{\ast}\phi} \right) \to \left( D^{\bullet}(M_{2})_{\C},\, \de_{\phi} \right) $$
is a morphism of differential $\Z$-graded complexes, and the map
$$ f_{\ast} \colon \left( D^{\bullet}(M_{1})_{\C},\, \del_{(f^{\ast}\theta_{1},f^{\ast}\theta_{2})},\, \delbar_{(f^{\ast}\theta_{1},f^{\ast}\theta_{2})} \right) \to \left( D^{\bullet}(M_{2})_{\C},\, \del_{(\theta_{1},\theta_{2})},\, \delbar_{(\theta_{1},\theta_{2})} \right) $$
is a morphism of bi-differential $\Z$-graded complexes.
In particular, $f$ induces the maps
\begin{eqnarray*}
f_{\ast}^{dR,\phi} &\colon& H^{\bullet}\left(D^{\bullet}(M_1)_{\C};\, \de_{f^\ast\phi}\right) \to H^{\bullet}\left(D^{\bullet}(M_2)_{\C};\, \de_{\phi}\right) \;, \\[5pt]
f_{\ast}^{\delbar,(\theta_1,\theta_2)} &\colon& H^{\bullet}\left(D^{\bullet}(M_1)_{\C};\, \delbar_{(f^*\theta_{1},f^*\theta_{2})}\right) \to H^{\bullet}\left(D^{\bullet}(M_2)_{\C};\, \delbar_{(\theta_{1},\theta_{2})}\right) \;, \\[5pt]
f_{\ast}^{BC,(\theta_1,\theta_2)} &\colon& H^{\bullet}\left(D^{\bullet}(M_1)_{\C};\, \del_{(f^\ast\theta_{1},f^\ast\theta_{2})},\, \delbar_{(f^\ast\theta_{1},f^\ast\theta_{2})};\, \del_{(f^\ast\theta_{1},f^\ast\theta_{2})}\delbar_{(f^\ast\theta_{1},f^\ast\theta_{2})}\right) \\[5pt]
&& \to H^{\bullet}\left(D^{\bullet}(M_2)_{\C};\, \del_{(\theta_{1},\theta_{2})},\, \delbar_{(\theta_{1},\theta_{2})};\, \del_{(\theta_{1},\theta_{2})}\delbar_{(\theta_{1},\theta_{2})}\right) \;, \\[5pt]
f_\ast^{A,(\theta_1,\theta_2)} &\colon& H^{\bullet}\left(D^{\bullet}(M_1)_{\C};\, \del_{(f^\ast\theta_{1},f^\ast\theta_{2})}\delbar_{(f^\ast\theta_{1},f^\ast\theta_{2})};\, \del_{(f^\ast\theta_{1},f^\ast\theta_{2})},\, \delbar_{(f^\ast\theta_{1},f^\ast\theta_{2})}\right) \\[5pt]
&& \to H^{\bullet}\left(D^{\bullet}(M_2)_{\C};\, \del_{(\theta_{1},\theta_{2})}\delbar_{(\theta_{1},\theta_{2})};\, \del_{(\theta_{1},\theta_{2})},\, \delbar_{(\theta_{1},\theta_{2})}\right) \;.
\end{eqnarray*}

\subsection{Hodge decomposition and \texorpdfstring{$\del\delbar$-Lemma}{partialoverlinepartial-Lemma} with twisted differentials}

As in \cite{deligne-griffiths-morgan-sullivan}, we consider the following definitions in the case of twisted differentials.

\begin{defi}
Let $\left( M,\, J \right)$ be a compact complex manifold of complex dimension $n$.
For $\theta_{1}, \theta_{2} \in H^{1,0}_{BC}(M)$, and $\phi:=\theta_{1}+\overline{\theta_{1}}+\theta_{2}-\overline{\theta_{2}}$, consider the bi-differential $\Z$-graded complex
\[ \left( A^{\bullet}(M)_{\C},\, \del_{(\theta_{1},\theta_{2})},\, \delbar_{(\theta_{1},\theta_{2})} \right) \;. \]

We say that $\left( M ,\, J \right)$:
\begin{enumerate}
\item satisfies the {\em $\del_{(\theta_{1},\theta_{2})}\delbar_{(\theta_{1},\theta_{2})}$-Lemma} if 
\[ \ker \del_{(\theta_{1},\theta_{2})}\cap \ker \delbar_{(\theta_{1},\theta_{2})}\cap ( \imm \del_{(\theta_{1},\theta_{2})}+\imm \delbar_{(\theta_{1},\theta_{2})}) \;=\; \imm \del_{(\theta_{1},\theta_{2})}\delbar_{(\theta_{1},\theta_{2})} \;, \]
i.e., if the natural map
\begin{eqnarray*}
\lefteqn{H^{\bullet}\left( A^{\bullet}(M)_{\C};\, \del_{(\theta_{1},\theta_{2})},\, \delbar_{(\theta_{1},\theta_{2})};\, \del_{(\theta_{1},\theta_{2})}\delbar_{(\theta_{1},\theta_{2})} \right)} \\[5pt]
&& \to H^{\bullet}\left( A^{\bullet}(M)_{\C};\, \del_{(\theta_{1},\theta_{2})} \delbar_{(\theta_{1},\theta_{2})};\, \del_{(\theta_{1},\theta_{2})},\, \delbar_{(\theta_{1},\theta_{2})} \right)
\end{eqnarray*}
induced by the identity is injective;

\item admits the {\em $(\theta_{1}, \theta_{2})$-Hodge decomposition} if the natural maps
 \[ H^{\bullet}\left( A^{\bullet}(M)_{\C};\, \del_{(\theta_{1},\theta_{2})},\, \delbar_{(\theta_{1},\theta_{2})};\, \del_{(\theta_{1},\theta_{2})}\delbar_{(\theta_{1},\theta_{2})} \right) \to H^{\bullet}\left( A^{\bullet}(M)_{\C};\, \de_{\phi}\right) \]
and
 \[ H^{\bullet} \left( A^{\bullet}(M)_{\C};\, \del_{(\theta_{1},\theta_{2})},\, \delbar_{(\theta_{1},\theta_{2})};\, \del_{(\theta_{1},\theta_{2})}\delbar_{(\theta_{1},\theta_{2})} \right) \to H^{\bullet} \left( A^{\bullet}(M)_{\C};\, \del_{(\theta_{1},\theta_{2})} \right) \]  
and
 \[ H^{\bullet} \left( A^{\bullet}(M)_{\C};\, \del_{(\theta_{1},\theta_{2})},\, \delbar_{(\theta_{1},\theta_{2})};\, \del_{(\theta_{1},\theta_{2})}\delbar_{(\theta_{1},\theta_{2})} \right) \to H^{\bullet} \left( A^{\bullet}(M)_{\C};\, \delbar_{(\theta_{1},\theta_{2})} \right) \]
induced by the identity are isomorphisms.
\end{enumerate}
\end{defi}

\medskip

By \cite[Lemma 5.15]{deligne-griffiths-morgan-sullivan}, see also \cite[Lemma 1.4]{angella-tomassini-5}, we have the following equivalent characterizations of $\del_{(\theta_1,\theta_2)}\delbar_{(\theta_1,\theta_2)}$-Lemma. (In case of double complex, a further characterization is proven in \cite{angella-tomassini-3}.)

\begin{lem}\label{lemma:deldelbar-equiv-def}
Let $\left( M,\, J \right)$ be a compact complex manifold of complex dimension $n$.
Take $\theta_{1}, \theta_{2} \in H^{1,0}_{BC}(M)$, and $\phi:=\theta_{1}+\overline{\theta_{1}}+\theta_{2}-\overline{\theta_{2}}$.
The following conditions are equivalent:
 \begin{itemize}
  \item $\left( M,\, J \right)$ satisfies the $\del_{(\theta_{1},\theta_{2})}\delbar_{(\theta_{1},\theta_{2})}$-Lemma, i.e., the natural map
    \begin{eqnarray*}
     \lefteqn{H^{\bullet}\left( A^{\bullet}(M)_{\C};\, \del_{(\theta_{1},\theta_{2})},\, \delbar_{(\theta_{1},\theta_{2})};\, \del_{(\theta_{1},\theta_{2})}\delbar_{(\theta_{1},\theta_{2})} \right)} \\[5pt]
     && \to H^{\bullet}\left( A^{\bullet}(M)_{\C};\, \del_{(\theta_{1},\theta_{2})} \delbar_{(\theta_{1},\theta_{2})};\, \del_{(\theta_{1},\theta_{2})},\, \delbar_{(\theta_{1},\theta_{2})} \right)
    \end{eqnarray*}
    induced by the identity is injective;
  \item the natural map
    \begin{eqnarray*}
     \lefteqn{H^{\bullet}\left( A^{\bullet}(M)_{\C};\, \del_{(\theta_{1},\theta_{2})},\, \delbar_{(\theta_{1},\theta_{2})};\, \del_{(\theta_{1},\theta_{2})}\delbar_{(\theta_{1},\theta_{2})} \right)} \\[5pt]
     && \to H^{\bullet}\left( A^{\bullet}(M)_{\C};\, \del_{(\theta_{1},\theta_{2})} \delbar_{(\theta_{1},\theta_{2})};\, \del_{(\theta_{1},\theta_{2})},\, \delbar_{(\theta_{1},\theta_{2})} \right)
    \end{eqnarray*}
    induced by the identity is an isomorphism;
  \item the natural maps
    \[H^{\bullet}(A^{\bullet}(M)_{\C};\, \del_{(\theta_{1},\theta_{2})},\, \delbar_{(\theta_{1},\theta_{2})};\, \del_{(\theta_{1},\theta_{2})}\delbar_{(\theta_{1},\theta_{2})})\to H^{\bullet}(A^{\bullet}(M)_{\C};\, \del_{(\theta_{1},\theta_{2})}) \]  
    and 
    \[H^{\bullet}(A^{\bullet}(M)_{\C};\, \del_{(\theta_{1},\theta_{2})},\, \delbar_{(\theta_{1},\theta_{2})};\, \del_{(\theta_{1},\theta_{2})}\delbar_{(\theta_{1},\theta_{2})})\to H^{\bullet}(A^{\bullet}(M)_{\C};\, \delbar_{(\theta_{1},\theta_{2})}) \] 
    induced by the identity are injective;
  \item the natural maps
    \[ H^{\bullet}(A^{\bullet}(M)_{\C};\, \del_{(\theta_{1},\theta_{2})}) \to H^{\bullet}(A^{\bullet}(M)_{\C};\, \del_{(\theta_{1},\theta_{2})}\delbar_{(\theta_{1},\theta_{2})};\, \del_{(\theta_{1},\theta_{2})},\, \delbar_{(\theta_{1},\theta_{2})}) \]  
    and 
    \[ H^{\bullet}(A^{\bullet}(M)_{\C};\, \delbar_{(\theta_{1},\theta_{2})}) \to H^{\bullet}(A^{\bullet}(M)_{\C};\, \del_{(\theta_{1},\theta_{2})}\delbar_{(\theta_{1},\theta_{2})};\, \del_{(\theta_{1},\theta_{2})},\, \delbar_{(\theta_{1},\theta_{2})}) \] 
    induced by the identity are surjective.
 \end{itemize}
 Furthermore, they imply the following conditions:
 \begin{itemize}
  \item the natural map
   \[H^{\bullet}(A^{\bullet}(M)_{\C};\, \del_{(\theta_{1},\theta_{2})},\, \delbar_{(\theta_{1},\theta_{2})};\, \del_{(\theta_{1},\theta_{2})}\delbar_{(\theta_{1},\theta_{2})})\to H^{\bullet}(A^{\bullet}(M)_{\C};\, \de_{\phi}) \]
   induced by the identity in injective;
  \item the natural map
   \[ H^{\bullet}(A^{\bullet}(M)_{\C};\, \de_{\phi}) \to H^{\bullet}(A^{\bullet}(M)_{\C};\, \del_{(\theta_{1},\theta_{2})}\delbar_{(\theta_{1},\theta_{2})};\, \del_{(\theta_{1},\theta_{2})},\, \delbar_{(\theta_{1},\theta_{2})}) \]
   induced by the identity in surjective.
 \end{itemize}
\end{lem}

In particular, we have that admitting the $(\theta_{1}, \theta_{2})$-Hodge decomposition is a stronger condition than satisfying the $\del_{(\theta_{1},\theta_{2})}\delbar_{(\theta_{1},\theta_{2})}$-Lemma. We wonder whether it is stricly stronger, namely, whether there exists an example of a compact complex manifold $M$ satisfying the $\del_{(\theta_{1},\theta_{2})}\delbar_{(\theta_{1},\theta_{2})}$-Lemma but not admitting the $(\theta_{1}, \theta_{2})$-Hodge decomposition, for some $\theta_1,\theta_2\in H^{1,0}_{BC}(M)$.

\medskip

In the K\"ahler case, we can summarize Theorem \ref{thm:kahler-deldelbar-tw} and Theorem \ref{thm:kahler-hodge-tw} in the following.

\begin{cor}\label{cor:kahler-deldelbar-hodge-tw}
 Let $\left( M ,\, J \right)$ be a compact complex manifold endowed with a K\"ahler metric.
 Take $\theta_{1}, \theta_{2} \in H^{1,0}_{BC}(M)$.
 Then $\left( M ,\, J \right)$ satisfies the $\del_{(\theta_1,\theta_2)}\delbar_{(\theta_1,\theta_2)}$-Lemma and admits the $(\theta_1,\theta_2)$-Hodge decomposition.
\end{cor}

\subsection{Modifications and cohomologies with twisted differentials}

We recall that a {\em modification} of a compact complex manifold $\left( M ,\, J \right)$ of complex dimension $n$ is a holomorphic map
$$ \mu \colon \left( \tilde M ,\, \tilde J \right) \to \left( M ,\, J \right) $$
such that:
\begin{itemize}
\item $\left( \tilde M ,\, \tilde J \right)$ is a compact $n$-dimensional complex manifold;
\item there exists an analytic subset $S \subset M$ of codimension greater than or equal to $1$ such that $\mu\lfloor_{\tilde M\backslash \mu^{-1}(S)} \colon \tilde M\backslash \mu^{-1}(S) \to  M\backslash S$ is a biholomorphism.
\end{itemize}

For the following, compare \cite[Theorem 3.1]{wells}.

\begin{thm}\label{thm:inj-modification-cohom}
 Let $\mu \colon \left( \tilde M ,\, \tilde J \right) \to \left( M ,\, J \right)$ be a proper modification of a compact complex manifold $\left( M,\, J \right)$.
 Take $\theta_{1}, \theta_{2} \in H^{1,0}_{BC}(M)$, and $\phi:=\theta_{1}+\overline{\theta_{1}}+\theta_{2}-\overline{\theta_{2}}$.
 Then the map $\mu$ induces the injective maps
 \begin{eqnarray*}
  \lefteqn{ \mu_{dR,\phi}^* \colon H^{\bullet}\left(A^{\bullet}(M)_{\C};\, \de_{\phi}\right) \to H^{\bullet}\left(A^{\bullet}(\tilde M)_{\C};\, \de_{\mu^\ast\phi}\right) \;, } \\[5pt]
  \lefteqn{ \mu_{\del,(\theta_{1},\theta_{2})}^* \colon H^{\bullet}\left(A^{\bullet}(M)_{\C};\, \del_{(\theta_{1},\theta_{2})}\right) \to H^{\bullet}\left(A^{\bullet}(\tilde M)_{\C};\, \del_{(\mu^\ast\theta_{1},\mu^\ast\theta_{2})}\right) \;, } \\[5pt]
  \lefteqn{ \mu_{\delbar,(\theta_{1},\theta_{2})}^* \colon H^{\bullet}\left(A^{\bullet}(M)_{\C};\, \delbar_{(\theta_{1},\theta_{2})}\right) \to H^{\bullet}\left(A^{\bullet}(\tilde M)_{\C};\, \delbar_{(\mu^\ast\theta_{1},\mu^\ast\theta_{2})}\right) \;, } \\[5pt]
  \lefteqn{\mu_{BC,(\theta_{1},\theta_{2})}^* \colon H^{\bullet}\left(A^{\bullet}(M)_{\C};\, \del_{(\theta_{1},\theta_{2})},\, \delbar_{(\theta_{1},\theta_{2})};\, \del_{(\theta_{1},\theta_{2})}\delbar_{(\theta_{1},\theta_{2})}\right)} \\[5pt]
  && \to H^{\bullet}\left(A^{\bullet}(\tilde M)_{\C};\, \del_{(\mu^\ast\theta_{1},\mu^\ast\theta_{2})},\, \delbar_{(\mu^\ast\theta_{1},\mu^\ast\theta_{2})};\, \del_{(\mu^\ast\theta_{1},\mu^\ast\theta_{2})}\delbar_{(\mu^\ast\theta_{1},\mu^\ast\theta_{2})}\right) \;, \\[5pt]
  \lefteqn{\mu_{A,(\theta_{1},\theta_{2})}^* \colon H^{\bullet}\left(A^{\bullet}(M)_{\C};\, \del_{(\theta_{1},\theta_{2})}\delbar_{(\theta_{1},\theta_{2})};\, \del_{(\theta_{1},\theta_{2})},\, \delbar_{(\theta_{1},\theta_{2})}\right)} \\[5pt]
  && \to H^{\bullet}\left(A^{\bullet}(\tilde M)_{\C};\, \del_{(\mu^\ast\theta_{1},\mu^\ast\theta_{2})}\delbar_{(\mu^\ast\theta_{1},\mu^\ast\theta_{2})};\, \del_{(\mu^\ast\theta_{1},\mu^\ast\theta_{2})},\, \delbar_{(\mu^\ast\theta_{1},\mu^\ast\theta_{2})}\right) \;.
 \end{eqnarray*}
 and the surjective maps
  \begin{eqnarray*}
  \lefteqn{ \mu^{dR,\phi}_* \colon H^{\bullet}\left(A^{\bullet}(\tilde M)_{\C};\, \de_{\mu^\ast\phi}\right) \to H^{\bullet}\left(A^{\bullet}(M)_{\C};\, \de_{\phi}\right) \;, } \\[5pt]
  \lefteqn{ \mu^{\del,(\theta_{1},\theta_{2})}_* \colon H^{\bullet}\left(A^{\bullet}(\tilde M)_{\C};\, \del_{(\mu^\ast\theta_{1},\mu^\ast\theta_{2})}\right) \to H^{\bullet}\left(A^{\bullet}(M)_{\C};\, \del_{(\theta_{1},\theta_{2})}\right) \;, } \\[5pt]
  \lefteqn{ \mu^{\delbar,(\theta_{1},\theta_{2})}_* \colon H^{\bullet}\left(A^{\bullet}(\tilde M)_{\C};\, \delbar_{(\mu^\ast\theta_{1},\mu^\ast\theta_{2})}\right) \to H^{\bullet}\left(A^{\bullet}(M)_{\C};\, \delbar_{(\theta_{1},\theta_{2})}\right) \;, } \\[5pt]
  \lefteqn{\mu^{BC,(\theta_{1},\theta_{2})}_* \colon H^{\bullet}\left(A^{\bullet}(\tilde M)_{\C};\, \del_{(\mu^\ast\theta_{1},\mu^\ast\theta_{2})},\, \delbar_{(\mu^\ast\theta_{1},\mu^\ast\theta_{2})};\, \del_{(\mu^\ast\theta_{1},\mu^\ast\theta_{2})}\delbar_{(\mu^\ast\theta_{1},\mu^\ast\theta_{2})}\right)} \\[5pt]
  && \to H^{\bullet}\left(A^{\bullet}(M)_{\C};\, \del_{(\theta_{1},\theta_{2})},\, \delbar_{(\theta_{1},\theta_{2})};\, \del_{(\theta_{1},\theta_{2})}\delbar_{(\theta_{1},\theta_{2})}\right) \;, \\[5pt]
  \lefteqn{\mu^{A,(\theta_{1},\theta_{2})}_* \colon H^{\bullet}\left(A^{\bullet}(\tilde M)_{\C};\, \del_{(\mu^\ast\theta_{1},\mu^\ast\theta_{2})}\delbar_{(\mu^\ast\theta_{1},\mu^\ast\theta_{2})};\, \del_{(\mu^\ast\theta_{1},\mu^\ast\theta_{2})},\, \delbar_{(\mu^\ast\theta_{1},\mu^\ast\theta_{2})}\right)} \\[5pt]
  && \to H^{\bullet}\left(A^{\bullet}(M)_{\C};\, \del_{(\theta_{1},\theta_{2})}\delbar_{(\theta_{1},\theta_{2})};\, \del_{(\theta_{1},\theta_{2})},\, \delbar_{(\theta_{1},\theta_{2})}\right) \;.
 \end{eqnarray*}
\end{thm}

\begin{proof}
 We follow closely the proof by R.~O. Wells in \cite[Theorem 3.1]{wells}.
 
 Consider the diagram
 $$ \xymatrix{
  A^{\bullet}(\tilde M)_\C \ar[r]^{T} & D^{\bullet}(\tilde M)_\C \ar[d]^{\mu_*} \\
  A^{\bullet}(M)_\C \ar[u]^{\mu^*} \ar[r]_{T} & D^{\bullet}(M)_{\C} \;.
 } $$

 There exists a proper analytic subset $\tilde S \subset \tilde M$ such that
 $$ \mu \lfloor_{\tilde M \setminus \tilde S} \colon \tilde M \setminus \tilde S \to M \setminus \mu(S) $$
 is a finitely-sheeted covering map of sheeting number $\ell\in\N\setminus\{0\}$. Let $\mathcal{U}:=\left\{U_\alpha\right\}_{j\in J}$ be an open covering of $M\setminus\mu(S)$, and let $\left\{\rho_\alpha\right\}_{j\in J}$ be an associated partition of unity. For every $\alpha,\beta\in A^{\bullet}(M)_\C$, one has that
 \begin{eqnarray*}
  \left\langle \mu_*T\mu^* \alpha,\beta \right\rangle
  &=& \left\langle T\mu^*\alpha,\mu^*\beta \right\rangle
  \;=\; \int_{\tilde M} \mu^*\alpha\wedge\mu^*\beta
  \;=\; \int_{\tilde M} \mu^*\left(\alpha\wedge\beta\right)
  \;=\; \int_{\tilde M \setminus \tilde S} \mu^*\left(\alpha\wedge\beta\right) \\[5pt]
  &=& \sum_{j\in J} \int_{\mu^{-1}(U_j)} \mu^*\left(\rho_j\cdot\left(\alpha\wedge\beta\right)\right)
  \;=\; \sum_{j\in J} \sum_{\sharp\left\{U \in \mathcal{U} \st \mu(U)=\mu(U_j)\right\}} \int_{U_j} \rho_j\cdot\left(\alpha\wedge\beta\right) \\[5pt]
  &=& \ell \cdot \sum_{j\in J} \int_{U_j} \rho_j\cdot\left(\alpha\wedge\beta\right) \;=\; \ell \cdot \int_{ M \setminus \mu(\tilde S) } \alpha\wedge\beta \;=\; \ell \cdot \int_{M} \alpha\wedge\beta \;=\; \left\langle \ell \cdot T \alpha , \beta \right\rangle \;,
 \end{eqnarray*}
 and hence one gets that
 $$ \mu_*T\mu^* \;=\; \ell \cdot T \;. $$

 In particular, one gets, for $\sharp\in\left\{\del,\delbar,BC,A\right\}$,
 $$ \mu_*^{dR,\phi} T \mu^*_{dR,\phi} \;=\; \ell \cdot T \qquad \text{ and } \qquad \mu_*^{\sharp,(\theta_1,\theta_2)} T \mu^*_{\sharp,(\theta_1,\theta_2)} \;=\; \ell \cdot T \;. $$
 Hence, in particular, for $\sharp\in\left\{\del,\delbar,BC,A\right\}$, the maps $\mu^*_{dR,\phi}$ and $\mu^*_{\sharp,(\theta_1,\theta_2)}$ are injective, and the maps $\mu_*^{dR,\phi}$ and $\mu_*^{\sharp,(\theta_1,\theta_2)}$ are surjective.
\end{proof}

\subsection{Modifications and \texorpdfstring{$\del\delbar$-Lemma}{partialoverlinepartial-Lemma} with twisted differentials}
As a consequence, we get the following two results. The first one concerns the behaviour of $\del_{(\theta_{1},\theta_{2})}\delbar_{(\theta_{1},\theta_{2})}$-Lemma under proper modifications.

\begin{thm}\label{thm:deldelbar-modification}
 Let $\mu \colon \left( \tilde M ,\, \tilde J \right) \to \left( M ,\, J \right)$ be a proper modification of a compact complex manifold $\left( M,\, J \right)$.
 Take $\theta_{1}, \theta_{2} \in H^{1,0}_{BC}(M)$. If $\left( \tilde M ,\, \tilde J \right)$ satisfies the $\del_{(\mu^{\ast}\theta_{1},\mu^{\ast}\theta_{2})}\delbar_{(\mu^{\ast}\theta_{1},\mu^{\ast}\theta_{2})}$-Lemma, then $\left( M ,\, J \right)$ satisfies the $\del_{(\theta_{1},\theta_{2})}\delbar_{(\theta_{1},\theta_{2})}$-Lemma.
\end{thm}

\begin{proof}
 Consider the commutative diagram
 $$ \xymatrix{
      H^{\bullet}_{BC}\left(M;(\theta_1,\theta_2)\right) \ar[d]_{\id_{M}^\ast} \ar[rr]^{\mu^*_{BC,(\theta_1,\theta_2)}} && H^{\bullet}_{BC}\left(\tilde M;(\mu^\ast\theta_1,\mu^\ast\theta_2)\right) \ar[d]^{\id_{\tilde M}^\ast} \\
      H^{\bullet}_{A}\left(M;(\theta_1,\theta_2)\right) \ar[rr]_{\mu^*_{A,(\theta_1,\theta_2)}} && H^{\bullet}_{A}\left(\tilde M;(\mu^\ast\theta_1,\mu^\ast\theta_2)\right) \\
 } $$
 where, for simplicity, we have denoted, e.g.,
 $$
 H^{\bullet}_{BC}\left(\tilde M;(\mu^\ast\theta_1,\mu^\ast\theta_2)\right) \;:=\; H^{\bullet}\left( A^{\bullet}(\tilde M)_{\C};\, \del_{(\mu^\ast\theta_{1},\mu^\ast\theta_{2})},\, \delbar_{(\mu^\ast\theta_{1},\mu^\ast\theta_{2})};\, \del_{(\mu^\ast\theta_{1},\mu^\ast\theta_{2})}\delbar_{(\mu^\ast\theta_{1},\mu^\ast\theta_{2})} \right)
 $$
 and
 $$
 H^{\bullet}_{A}\left(\tilde M;(\mu^\ast\theta_1,\mu^\ast\theta_2)\right) \;:=\; H^{\bullet}\left( A^{\bullet}(\tilde M)_{\C};\, \del_{(\mu^\ast\theta_{1},\mu^\ast\theta_{2})}\delbar_{(\mu^\ast\theta_{1},\mu^\ast\theta_{2})};\, \del_{(\mu^\ast\theta_{1},\mu^\ast\theta_{2})},\, \delbar_{(\mu^\ast\theta_{1},\mu^\ast\theta_{2})} \right) \;.
 $$

 Suppose that $\left( \tilde M ,\, \tilde J \right)$ satisfies the $\del_{(\mu^\ast\theta_1,\mu^\ast\theta_2)}\delbar_{(\mu^\ast\theta_1,\mu^\ast\theta_2)}$-Lemma. Then, by definition, the map $\id_{\tilde M}^\ast$ is injective. Furthermore, the map $\mu^*_{BC,(\theta_1,\theta_2)}$ is injective by Theorem \ref{thm:inj-modification-cohom}. Hence the map $\id_{M}^\ast$ is injective, that is, $\left( M ,\, J \right)$ satisfies the $\del_{(\theta_1,\theta_2)}\delbar_{(\theta_1,\theta_2)}$-Lemma.
\end{proof}

\subsection{Modifications and Hodge decomposition with twisted differentials}
The second result concerns Hodge decomposition with twisted differential.

\begin{thm}\label{thm:hodgedec-modification}
 Let $\mu \colon \left( \tilde M ,\, \tilde J \right) \to \left( M ,\, J \right)$ be a proper modification of a compact complex manifold $\left( M,\, J \right)$.
 Take $\theta_{1}, \theta_{2} \in H^{1,0}_{BC}(M)$. If $\left( \tilde M ,\, \tilde J \right)$ satisfies the $\left(\mu^{\ast}\theta_{1},\mu^{\ast}\theta_{2}\right)$-Hodge decomposition, then $\left( M ,\, J \right)$ satisfies the $\left(\theta_{1},\theta_{2}\right)$-Hodge decomposition.
\end{thm}

\begin{proof}
 For simplicity, we denote, e.g.,
 $$
 H^{\bullet}_{BC}\left(\tilde M;(\mu^\ast\theta_1,\mu^\ast\theta_2)\right) \;:=\; H^{\bullet}\left( A^{\bullet}(\tilde M)_{\C};\, \del_{(\mu^\ast\theta_{1},\mu^\ast\theta_{2})},\, \delbar_{(\mu^\ast\theta_{1},\mu^\ast\theta_{2})};\, \del_{(\mu^\ast\theta_{1},\mu^\ast\theta_{2})}\delbar_{(\mu^\ast\theta_{1},\mu^\ast\theta_{2})} \right)
 $$
 and
 $$
 H^{\bullet}_{A}\left(\tilde M;(\mu^\ast\theta_1,\mu^\ast\theta_2)\right) \;:=\; H^{\bullet}\left( A^{\bullet}(\tilde M)_{\C};\, \del_{(\mu^\ast\theta_{1},\mu^\ast\theta_{2})}\delbar_{(\mu^\ast\theta_{1},\mu^\ast\theta_{2})};\, \del_{(\mu^\ast\theta_{1},\mu^\ast\theta_{2})},\, \delbar_{(\mu^\ast\theta_{1},\mu^\ast\theta_{2})} \right) \;.
 $$

 Consider the commutative diagram
 $$ \xymatrix{
      H^{\bullet}_{BC}\left(M;(\theta_1,\theta_2)\right) \ar[d]_{\id_{M}^\ast} \ar[rr]^{\mu^*_{BC,(\theta_1,\theta_2)}} && H^{\bullet}_{BC}\left(\tilde M;(\mu^\ast\theta_1,\mu^\ast\theta_2)\right) \ar[d]^{\id_{\tilde M}^\ast} \\
      H^{\bullet}_{dR}\left(M;\phi\right) \ar[rr]_{\mu^*_{dR,\phi}} && H_{dR}\left(\tilde M;\mu^\ast\phi\right) \\
 } .$$
 Then by Theorem \ref{thm:inj-modification-cohom}, $\mu^*_{BC,(\theta_1,\theta_2)}$ and $\mu^*_{dR,\phi}$ are injective.
 Hence by the injectivity of $\id_{\tilde M}^\ast$, the map $\id_{M}^\ast:H^{\bullet}_{BC}\left(M;(\theta_1,\theta_2)\right)\to H^{\bullet}_{dR}\left(M;\phi\right) $ is injective.

 Considering the direct image maps, we have the commutative diagram
 $$ \xymatrix{
      H^{\bullet}_{BC}\left(M;(\theta_1,\theta_2)\right) \ar[d]_{\id_{M}^\ast} && \ar[ll]^{\mu_*^{BC,(\theta_1,\theta_2)}} H^{\bullet}_{BC}\left(\tilde M;(\mu^\ast\theta_1,\mu^\ast\theta_2)\right) \ar[d]^{\id_{\tilde M}^\ast} \\
      H^{\bullet}_{dR}\left(M;\phi\right)  &&\ar[ll]_{\mu_*^{dR,\phi}} H_{dR}\left(\tilde M;\mu^\ast\phi\right) \\
 } .$$
 By Theorem \ref{thm:inj-modification-cohom}, $\mu_*^{BC,(\theta_1,\theta_2)}$ and $\mu_*^{dR,\phi}$ are surjective.
 By the surjectivity of $\id_{\tilde M}^\ast$,  the map $\id_{M}^\ast:H^{\bullet}_{BC}\left(M;(\theta_1,\theta_2)\right)\to H^{\bullet}_{dR}\left(M;\phi\right) $ is surjective.

 Arguing in the same way with the Dolbeault cohomologies, we get that $M$ satisfies the $\left(\theta_1,\theta_2\right)$-Hodge decomposition.
\end{proof}

\medskip

We recall that a compact complex manifold $\left( M,\, J \right)$ is said to be in {\em class $\mathcal{C}$ of Fujiki}, \cite{fujiki}, if it admits a proper modification $\mu \colon \left( \tilde M ,\, \tilde J \right) \to \left( M ,\, J \right)$ with $\left( \tilde M ,\, \tilde J \right)$ admitting K\"ahler metrics. In particular, a {\em Mo\v\i\v shezon manifold}, \cite{moishezon}, (that is, a compact complex manifold of complex dimension $n$ such that the degree of transcendence over $\C$ of the field of meromorphic functions is equal to $n$,) admits a proper modification from a projective manifold, \cite[Theorem 1]{moishezon}, and therefore belongs to class $\mathcal{C}$ of Fujiki.

\begin{cor}\label{cor:classC-hodge}
 Let $\left( M ,\, J \right)$ be a compact complex manifold in class $\mathcal{C}$ of Fujiki.
 Take $\theta_{1}, \theta_{2} \in H^{1,0}_{BC}(M)$.
 Then $\left( M ,\, J \right)$ satisfies the $\del_{(\theta_{1},\theta_{2})}\delbar_{(\theta_{1},\theta_{2})}$-Lemma and admits the $\left(\theta_1,\theta_2\right)$-Hodge decomposition.
\end{cor}

\begin{proof}
 Consider a proper modification $\mu \colon \left( \tilde M ,\, \tilde J \right) \to \left( M ,\, J \right)$ with $\left( \tilde M ,\, \tilde J \right)$ admitting K\"ahler metrics. From Corollary \ref{cor:kahler-deldelbar-hodge-tw}, (see also Theorem \ref{thm:kahler-deldelbar-tw} and Theorem \ref{thm:kahler-hodge-tw},) the compact K\"ahler manifold $\left( \tilde M ,\, \tilde J \right)$ satisfies the $\del_{(\mu^\ast\theta_1,\mu^\ast\theta_2)}\delbar_{(\mu^\ast\theta_1,\mu^\ast\theta_2)}$-Lemma, and admits the $(\mu^\ast\theta_1,\mu^\ast\theta_2)$-Hodge decomposition. Therefore, from Theorem \ref{thm:deldelbar-modification} and Theorem \ref{thm:hodgedec-modification}, the compact complex manifold $\left( M ,\, J \right)$ satisfies the $\del_{(\theta_1,\theta_2)}\delbar_{(\theta_1,\theta_2)}$-Lemma and admits the $(\theta_1,\theta_2)$-Hodge decomposition.
\end{proof}

\subsection{Complex orbifolds of global-quotient type}
I. Satake introduce in \cite{satake} the notion of {\em orbifold}, also called \emph{V-manifold}; see also \cite{baily, baily-2}.
It is a singular complex space whose singularities are locally isomorphic to quotient singularities $\left. \C^n \middle\slash G \right.$, for finite subgroups $G \subset \GL_n(\C)$.
In particular, we are interested in compact {\em complex orbifolds of global-quotient type}, namely, compact complex orbifolds given by $\hat M = \left. M \middle\slash G \right.$ where $M$ is a compact complex manifold and $G$ is a finite group of biholomorphisms of $M$. See \cite{angella-2} and the references therein for motivations.

\medskip

From the cohomological point of view, one can adapt both the sheaf-theoretic and the analytic tools to complex orbifolds, see \cite{satake, baily, baily-2, angella-2}. In particular, let $\hat M = \left. M \middle\slash G \right.$ be a compact complex orbifold of global-quotient type. Consider the double-complex $\left( \wedge^{\bullet,\bullet} \hat M ,\, \del ,\, \delbar \right)$, where the space $\wedge^{\bullet,\bullet} \hat M$ of differential forms on $\hat M$ is defined as the space of $G$-invariant differential forms on $M$. Consider the associated cohomologies. Fix a Hermitian metric on $\hat M$, namely, a $G$-invariant Hermitian metric on $M$. Consider the Laplacian operators defined as in the smooth case. Then Hodge theory applies, \cite[Theorem H, Theorem K]{baily}, \cite[Theorem 1]{angella-2}.

\medskip

By considering objects on $\hat M$ as $G$-invariant objects on $M$, one can adapt all the definitions and results in the previous sections in a straightforward way. In particular, as an analogue of \cite[Theorem 2]{angella-2}, we can restate Theorem \ref{thm:deldelbar-modification} and Theorem \ref{thm:hodgedec-modification} as follows.

\begin{thm}\label{thm:mod-orbfd-glob}
 Let $\mu \colon \hat N \to \hat M$ be a proper modification between compact complex orbifolds of global-quotient type.
 Take $\theta_{1}, \theta_{2} \in H^{1,0}_{BC}(\hat M)$.
 \begin{itemize}
  \item If $\hat N$ satisfies the $\del_{(\mu^{\ast}\theta_{1},\mu^{\ast}\theta_{2})}\delbar_{(\mu^{\ast}\theta_{1},\mu^{\ast}\theta_{2})}$-Lemma, then $\hat M$ satisfies the $\del_{(\theta_{1},\theta_{2})}\delbar_{(\theta_{1},\theta_{2})}$-Lemma.
  \item If $\hat N$ satisfies both the $\left(\mu^{\ast}\theta_{1},\mu^{\ast}\theta_{2}\right)$-Hodge decomposition, %and the $\left(-\mu^{\ast}\overline\theta_{2},-\mu^{\ast}\overline\theta_{1}\right)$-Hodge decomposition, 
then $\hat M$ satisfies the $\left(\theta_{1},\theta_{2}\right)$-Hodge decomposition.
 \end{itemize}
\end{thm}

Therefore, we have the following corollary. (As usual, by compact complex orbifold $\hat M$ of global-quotient type in class $\mathcal{C}$ of Fujiki, we mean that there exists a proper modification $\mu \colon \hat N \to \hat M$ where $\hat N$ is a compact complex orbifold of global-quotient type admitting K\"ahler metrics.)

\begin{cor}\label{cor:orbld-glob-class-c}
 Let $\hat M$ be a compact complex orbifold of global-quotient type in class $\mathcal{C}$ of Fujiki.
 Take $\theta_{1}, \theta_{2} \in H^{1,0}_{BC}(\hat M)$.
 Then $\hat M$ satisfies the $\del_{(\theta_{1},\theta_{2})}\delbar_{(\theta_{1},\theta_{2})}$-Lemma and admits the $\left(\theta_1,\theta_2\right)$-Hodge decomposition.
\end{cor}

\section{Solvmanifolds}

In this section, we consider solvmanifolds, i.e., compact quotients $\left. \Gamma \middle\backslash G \right.$ where $G$ is a connected simply-connected solvable Lie group and $\Gamma$ is a co-compact discrete subgroup.

\subsection{Cohomology computations for solvmanifolds}

Let $G$ be a connected simply-connected solvable Lie group endowed with a left-invariant complex structure $J$ and admitting a lattice $\Gamma$. Its associated Lie algebra is denoted by $\mathfrak{g}$, and its complexification by $\g_\C:=\g\otimes_\R\C$.
Then we consider the sub-double complex
$$ \left( \wedge^{\bullet,\bullet} \g^{\ast}_\C ,\, \del ,\, \delbar \right) \hookrightarrow \left( A^{\bullet,\bullet}(\solvmfd)_\C ,\, \del ,\, \delbar \right) \;. $$

Take $\theta_{1}, \theta_{2} \in H^{1,0}\left(\wedge^{\bullet,\bullet} \g^\ast_\C; \del,\delbar; \del\delbar \right) \hookrightarrow H^{1,0}\left(A^{\bullet,\bullet} (\solvmfd)_\C; \del,\delbar; \del\delbar \right)$. (For the injectivity, see \cite[Lemma 3.6]{angella-1}, see also \cite[Lemma 9]{console-fino}.)
Then we have the bi-differential $\Z$-graded sub-complex
$$ \left( \wedge^{\bullet} \g^{\ast}_\C,\, \del_{(\theta_{1},\theta_{2})},\, \delbar_{(\theta_{1},\theta_{2})} \right) \hookrightarrow \left( A^{\bullet}(\solvmfd)_\C,\, \del_{(\theta_{1},\theta_{2})},\, \delbar_{(\theta_{1},\theta_{2})} \right) \;. $$

\medskip

We firstly prove the following result, which generalizes \cite[Lemma 9]{console-fino} and \cite[Lemma 3.6]{angella-1} to the case of twisted differentials.
(Consider also the F.~A. Belgun symmetrization trick, \cite[Theorem 7]{belgun}, as a different argument.)

\begin{prop}\label{prop:inj-cohom-solvmfds}
Let $\solvmfd$ be a solvmanifold endowed with a $G$-left-invariant complex structure, and with associated Lie algebra $\g$. Take $\theta_{1}, \theta_{2} \in H^{1,0}\left(\wedge^\bullet \g^\ast_\C; \del,\delbar; \del\delbar \right)$, and $\phi:=\theta_{1}+\overline{\theta_{1}}+\theta_{2}-\overline{\theta_{2}}$.

The maps
\begin{eqnarray*}
 H^{\bullet}\left(\wedge^{\bullet} \g^{\ast}_\C;\, \de_{\phi}\right) &{\to} & H^{\bullet}\left(A^{\bullet}(\solvmfd)_{\C};\, \de_{\phi}\right) \;, \\[5pt]
 H^{\bullet}\left(\wedge^{\bullet} \g^{\ast}_\C;\, \del_{(\theta_{1},\theta_{2})}\right) & {\to} & H^{\bullet}\left(A^{\bullet}(\solvmfd)_{\C};\, \del_{(\theta_{1},\theta_{2})}\right) \;, \\[5pt]
 H^{\bullet}\left(\wedge^{\bullet} \g^{\ast}_\C;\, \delbar_{(\theta_{1},\theta_{2})}\right) & {\to} & H^{\bullet}\left(A^{\bullet}(\solvmfd)_{\C};\, \delbar_{(\theta_{1},\theta_{2})}\right) \;, \\[5pt]
 H^{\bullet}\left(\wedge^{\bullet} \g^{\ast}_\C;\, \del_{(\theta_{1},\theta_{2})},\, \delbar_{(\theta_{1},\theta_{2})};\, \del_{(\theta_{1},\theta_{2})}\delbar_{(\theta_{1},\theta_{2})}\right) &{\to} & H^{\bullet}\left(A^{\bullet}(\solvmfd)_{\C};\, \del_{(\theta_{1},\theta_{2})},\, \delbar_{(\theta_{1},\theta_{2})};\, \del_{(\theta_{1},\theta_{2})}\delbar_{(\theta_{1},\theta_{2})}\right) \;, \\[5pt]
 H^{\bullet}\left(\wedge^{\bullet} \g^{\ast}_\C;\, \del_{(\theta_{1},\theta_{2})}\delbar_{(\theta_{1},\theta_{2})};\, \del_{(\theta_{1},\theta_{2})},\, \delbar_{(\theta_{1},\theta_{2})}\right) & {\to} & H^{\bullet}\left(A^{\bullet}(\solvmfd)_{\C};\, \del_{(\theta_{1},\theta_{2})}\delbar_{(\theta_{1},\theta_{2})};\, \del_{(\theta_{1},\theta_{2})},\, \delbar_{(\theta_{1},\theta_{2})}\right) \; \\[5pt]
\end{eqnarray*}
induced by the inclusion $\left( \wedge^{\bullet} \g^{\ast}_\C,\, \del_{(\theta_{1},\theta_{2})},\, \delbar_{(\theta_{1},\theta_{2})} \right) \hookrightarrow \left( A^{\bullet}(\solvmfd)_\C,\, \del_{(\theta_{1},\theta_{2})},\, \delbar_{(\theta_{1},\theta_{2})} \right)$
are injective.
\end{prop}

\begin{proof}
Consider each case. Fix $g$ a $G$-left-invariant Hermitian metric on $\solvmfd$.
The metric $g$ and the forms $\theta_1$ and $\theta_2$ being left-invariant, the associated Laplacian operator $\Delta_{\sharp}^g$ satisfies $\left. \Delta_{\sharp}^g \right\lfloor_{\wedge^{\bullet}\g^\ast_\C} \colon \wedge^{\bullet}\g^\ast_\C \to \wedge^{\bullet}\g^\ast_\C$. In particular, Hodge theory applies both to $\left( \wedge^{\bullet} \g^{\ast}_\C,\, \del_{(\theta_{1},\theta_{2})},\, \delbar_{(\theta_{1},\theta_{2})} \right)$ and to $\left( A^{\bullet}(\solvmfd)_\C,\, \del_{(\theta_{1},\theta_{2})},\, \delbar_{(\theta_{1},\theta_{2})} \right)$. Hence we have the commutative diagram
$$
\xymatrix{
 \left. \Delta_{\sharp}^{g} \right\lfloor_{\wedge^{\bullet}\g^\ast_\C} \ar@{^{(}->}[r] \ar[d]_{\simeq} & \Delta_{\sharp}^{g} \ar[d]^{\simeq} \\
 H^{\bullet}_{\sharp} \left(\wedge^{\bullet} \g^{\ast}_\C\right) \ar[r] & H^{\bullet}_{\sharp} \left(A^{\bullet}(\solvmfd)_{\C}\right) \;,
}
$$
where $H^{\bullet}_{\sharp} \left(\wedge^{\bullet} \g^{\ast}_\C\right)$ and $H^{\bullet}_{\sharp} \left(A^{\bullet}(\solvmfd)_{\C}\right)$ denote the corresponding cohomologies. It yields the injectivity of the map $H^{\bullet}_{\sharp} \left(\wedge^{\bullet} \g^{\ast}_\C\right) \to H^{\bullet}_{\sharp} \left(A^{\bullet}(\solvmfd)_{\C}\right)$. Compare \cite[Proposition 2.2]{angella-kasuya-1}.
\end{proof}

\begin{ex}\label{ex:nakamura}
Take $G:=\C\ltimes _{\phi}\C^{2}$ where
$$ \phi(z_{1}) \;:=\; \left(
\begin{array}{cc}
\esp^{\frac{z_{1}+{\bar z_{1}}}{2}}& 0  \\
0&    \esp^{-\frac{z+{\bar z_{1}}}{2}}  
\end{array}
\right) \;\in\; \mathrm{GL}(\C^2) \;. $$
Then for some $a\in \R$  the matrix $\left(
\begin{array}{cc}
\esp^{a}& 0  \\
0&    \esp^{-a}  
\end{array}
\right)$
 is conjugate to an element of $\mathrm{SL}(2;\Z)$.
 Hence, for any $b\in \R\setminus\{0\}$, we have a lattice $\Gamma:=\left(a\, \Z+b\,\sqrt{-1}\,\Z \right)\ltimes \Gamma^{\prime\prime}$ of $G$, where $\Gamma^{\prime\prime} $ is a lattice of $\C^{2}$. The solvmanifold $\solvmfd$ is called {\em completely-solvable Nakamura manifold}, \cite[page 90]{nakamura}; see also, e.g., \cite[\S3]{debartolomeis-tomassini}, \cite[Example 1]{kasuya-mathz}, \cite[Example 2.17]{angella-kasuya-1}. If $b\not\in \pi\,\Z$, then $\solvmfd$ satisfies the $\del\delbar$-Lemma, see \cite[Example 2.17]{angella-kasuya-1} (see also \cite{kasuya-hodge}).

Consider local holomorphic coordinates $(z_{1},z_{2},z_{3})$ for $\C\ltimes _{\phi}\C^{2}$. We have
\[ \wedge ^{\bullet,\bullet}\g^{\ast}_{\C} \;=\;\wedge ^{\bullet,\bullet} \left( \left\langle \de z_{1},\, \esp^{-\frac{z_{1}+{\bar z_{1}}}{2}}\, \de z_{2},\, \esp^{\frac{z_{1}+{\bar z_{1}}}{2}}\, \de z_{3} \right\rangle \otimes \left\langle \de\bar z_{1},\, \esp^{-\frac{z_{1}+{\bar z_{1}}}{2}}\,\de\bar z_{2},\, \esp^{\frac{z_{1}+{\bar z_{1}}}{2}}\,\de\bar z_{3} \right\rangle \right) \;. \]

Take
$$ \theta_{1} \;:=\; \frac{\de z_1}{2} \qquad \text{ and } \qquad \theta_{2} \;:=\; 0 \;, $$
and set
$$ \phi \;:=\; \theta_{1}+\overline{\theta_{1}}+\theta_{2}-\overline{\theta_{2}} \;=\; \frac{\de z_1 + \de \bar z_1}{2} \;. $$
In \cite[\S8]{kasuya-hyper-hodge}, the second author computed
\[ H^{\bullet}(\solvmfd; \de_{\phi}) \;\neq\; \{0\} \]
and
\[ H^{\bullet}(\solvmfd; \delbar_{(\theta_1,\theta_2)}) \;=\; \{0\} \;. \]
Hence $\solvmfd$ does not admit the $(\theta_1,\theta_2)$-Hodge decomposition.

We show now that also the $\del_{(\theta_1,\theta_2)}\delbar_{(\theta_1,\theta_2)}$-Lemma does not hold on $\solvmfd$.
Consider 
\[\frac{1}{2}\,\esp^{\frac{z_{1}+{\bar z_{1}}}{2}}\, \left(\de z_{1}+\de\bar z_{1}\right)\wedge \de\bar z_{3} \;\in\; \wedge^{2} \g^{\ast}_\C \;. \]
We have
\[ 0 \;\neq\; \left[\frac{1}{2}\,\esp^{\frac{z_{1}+{\bar z_{1}}}{2}}\left(\de z_{1}+\de\bar z_{1}\right)\wedge \de\bar z_{3}\right] \;\in\; H^{2}\left(\wedge^{\bullet} \g^{\ast}_\C;\, \del_{(\theta_1,\theta_2)},\,\delbar_{(\theta_1,\theta_2)};\, \del_{(\theta_1,\theta_2)}\delbar_{(\theta_1,\theta_2)} \right) \;. \]
On the other hand, we have
\[ \frac{1}{2}\, \esp^{\frac{z_{1}+{\bar z_{1}}}{2}}\, \left( \de z_{1}+\de\bar z_{1}\right)\wedge \de\bar z_{3} \;=\; \delbar_{(\theta_1,\theta_2)}\left( \esp^{\frac{z_{1}+{\bar z_{1}}}{2}}\,\de \bar z_{3}\right) \;. \]
Therefore the map
\begin{eqnarray*}
 \lefteqn{ H^{2}\left(\wedge^{\bullet} \g^{\ast}_\C;\, \del_{(\theta_1,\theta_2)},\, \delbar_{(\theta_1,\theta_2)};\, \del_{(\theta_1,\theta_2)}\delbar_{(\theta_1,\theta_2)} \right) } \\[5pt]
 && \to H^{2}\left(A^{\bullet} (\solvmfd)_\C;\, \del_{(\theta_1,\theta_2)},\, \delbar_{(\theta_1,\theta_2)};\, \del_{(\theta_1,\theta_2)}\delbar_{(\theta_1,\theta_2)} \right) \\[5pt]
 && \to H^{2}\left(A^{\bullet} (\solvmfd)_\C;\, \del_{(\theta_1,\theta_2)}\delbar_{(\theta_1,\theta_2)};\, \del_{(\theta_1,\theta_2)},\, \delbar_{(\theta_1,\theta_2)} \right)
\end{eqnarray*}
is not injective. Since the first map is injective by Proposition \ref{prop:inj-cohom-solvmfds}, it follows that the natural map $H^{2}\left(A^{\bullet} (\solvmfd)_\C;\, \del_{(\theta_1,\theta_2)},\, \delbar_{(\theta_1,\theta_2)};\, \del_{(\theta_1,\theta_2)}\delbar_{(\theta_1,\theta_2)} \right) \to H^{2}\left(A^{\bullet} (\solvmfd)_\C;\, \del_{(\theta_1,\theta_2)}\delbar_{(\theta_1,\theta_2)};\, \del_{(\theta_1,\theta_2)},\, \delbar_{(\theta_1,\theta_2)} \right)$ induced by the identity is not injective.
It follows that $\solvmfd$ does not satisfy the $\del_{(\theta_1,\theta_2)}\delbar_{(\theta_1,\theta_2)}$-Lemma.
\end{ex}

\subsection{Solvmanifolds and \texorpdfstring{$\del\delbar$-Lemma}{partialoverlinepartial-Lemma} with twisted differentials}

The Weinstein and Thurston problem, concerning the characterization of nilmanifolds admitting K\"ahler structures, was solved by Ch. Benson and C.~S. Gordon, \cite[Theorem A]{benson-gordon-nilmanifolds}. In fact, in \cite[Theorem 1, Corollary]{hasegawa}, K. Hasegawa proved that non-tori nilmanifold are not formal in the sense of Sullivan, and hence do not belong to class $\mathcal{C}$ of Fujiki.

As regards the characterization of solvmanifolds admitting K\"ahler structure, K. Hasegawa proved the following in \cite[Main Theorem]{hasegawa-osaka}. Let $X$ be a compact homogeneous space of solvable Lie group, that is, a compact differentiable manifold on which a connected solvable Lie group acts transitively. Then $X$ admits a Kähler structure if and only if it is a finite quotient of a complex torus which has a structure of a complex torus-bundle over a complex torus. In particular, a completely-solvable solvmanifold has a K\"ahler structure if and only if it is a complex torus.

As regards solvmanifolds in class $\mathcal{C}$ of Fujiki, they are characterized in \cite[Theorem 9]{arapura} by D. Arapura. More precisely, it follows from \cite[Theorem 3, Theorem 9]{arapura} that, for solvmanifolds endowed with complex structures, the properties of admitting K\"ahler metrics and of belonging to class $\mathcal{C}$ of Fujiki are equivalent. The proof is sketched at \cite[page 136]{arapura}, and is based on the fact that a finitely-presented group is a Fujiki group if and only if it is a K\"ahler group, see also \cite[Theorem 1.1]{bharali-biswas-mj} by G. Bharali, I. Biswas, and M. Mj. In fact, their result founds on the Hironaka elimination of indeterminacies, \cite[\S2]{bharali-biswas-mj}. By using the results by the second author in \cite{kasuya-hyper-hodge} and the above results, we can provide a different and more direct proof, of cohomological flavour.

\begin{thm}\label{thm:classC-solvmanifolds}
 Let $\left( M ,\, J \right)$ be a solvmanifold endowed with a complex structure. If $\left( M ,\, J \right)$ belongs to class $\mathcal{C}$ of Fujiki, then it admits a K\"ahler metric.
\end{thm}

\begin{proof}
 Take any $\theta_{1}, \theta_{2} \in H^{1,0}_{BC}(M)$. By Corollary \ref{cor:classC-hodge}, the manifold $\left( M ,\, J \right)$ admits the $(\theta_1,\theta_2)$-Hodge decomposition. In \cite{kasuya-hyper-hodge}, the property of satisfying the Hodge-decomposition with respect to any $\theta_{1}, \theta_{2} \in H^{1,0}_{BC}(M)$ is called hyper-strong-Hodge-decomposition. The second author proved in \cite[Theorem 1.7]{kasuya-hyper-hodge} that a solvmanifold admitting hyper-strong-Hodge-decomposition admits a K\"ahler metric.
\end{proof}

\end{document}